\documentclass[11pt]{amsart}
\usepackage{amsmath,amssymb,amsfonts,amsthm, amscd,indentfirst}
\usepackage{amsmath,latexsym,amssymb,amsmath,%awheremscd,
amscd,amsthm,amsxtra}
\usepackage{amsrefs}
\usepackage{hyperref}
\usepackage{url}
\usepackage{xcolor}
\usepackage{pdfpages}
\usepackage{graphics}
\usepackage{enumitem}

\nocite{*}

\newtheorem{theorem}{Theorem}[section]
\newtheorem{lemma}{Lemma}[section]
\newtheorem{corollary}{Corollary}[section]

\newtheorem{proposition}{Proposition}[section]

\theoremstyle{definition}
\newtheorem{defn}{Definition}[section]
\newtheorem{remark}{\textbf{Remark}}[section]

\def\ss{\mathbb{S}}

\def\CC{\mathbb{C}}
\def\tr{\mathrm{trace}}

\def\ve{\varepsilon}

\def\lan{\langle}
\def\ran{\rangle}

\def\n{\nabla}

\def\vp{\varphi}
\def\R{{\mathbb R}}

\def\hess{\mathrm{Hess}}

\DeclareMathOperator{\vol}{vol}
\DeclareMathOperator{\dist}{dist}          
     
\numberwithin{equation} {section}

\begin{document}
\title[]{Hopf type Theorem for Self-Shrinkers}
%\author{XXX}

\begin{abstract}
   \end{abstract}

\date{}

\medskip

%\tableofcontents

\author{Hil\'ario Alencar, Greg\'orio Silva Neto \& Detang Zhou}

\dedicatory{Dedicated to the memory of Manfredo do Carmo}

\date{August 04, 2020}

\address{Instituto de Matem\' atica, Universidade Federal de Alagoas, Macei\'o, AL, 57072-900, Brazil}

\email{hilario@mat.ufal.br}

\address{Instituto de Matem\'atica, Universidade Federal de Alagoas, Macei\'o, AL, 57072-900, Brazil}

\email{gregorio@im.ufal.br}

\address{Instituto de Matem\'atica e Estat\'istica, Universidade Federal Fluminense, Niter\'oi, RJ, 24210-201, Brazil}

\email{zhoud@id.uff.br}

\subjclass[2010]{Primary 53C42; Secondary 53C44; 53A10}

\keywords{Self-shrinkers, constant weighted mean curvature, Gauss space, round sphere, Hopf differential, weighted manifolds}

\footnote{Hil\'ario Alencar, Greg\'orio Silva Neto and Detang Zhou were partially supported by the National Council for Scientific and Technological Development - CNPq of Brazil.}

\begin{abstract}
In this paper, we prove that a two-dimensional self-shrinker, homeomorphic to the sphere, immersed in the three dimensional Euclidean space $\R^3$ is a round sphere, provided its mean curvature and the norm of  its position vector have an upper bound in terms of the norm of its traceless second fundamental form.  The example constructed by Drugan justifies that the hypothesis on the second fundamental form is necessary.  We can also prove the same kind of rigidity results for surfaces with parallel weighted mean curvature vector in $\R^n$ with radial weight. These results are  applications of a new generalization of Cauchy's Theorem in complex analysis which concludes that a complex function  is identically zero or  its zeroes are isolated if it satisfies some  weak holomorphy.
\end{abstract}

\maketitle

\section{Introduction}

An immersion $X:\Sigma\to\R^3$ of a two-dimensional surface $\Sigma$ is called a self-shrinker for the mean curvature flow if its mean curvature vector ${\bf H}$  satisfies the equation
\[
{\bf H}=-\frac{1}{2}X^\perp,
\]
where $X^\perp$ is the normal part of the position vector.

Self-shrinkers are the self-similar solutions of the mean curvature flow and many efforts were made in the last decades in order to obtain examples of such surfaces and classify these surfaces under certain geometrical restrictions. In particular, there is a problem to classify the sphere as the only compact self-shrinker under some geometrical assumptions, following the same spirit of the classical Hopf and Alexandrov results. In 1951, see \cite{hopf} and \cite{hopf-2}, Hopf proved that the only surfaces of $\R^3,$ homeomorphic to the sphere, with constant mean curvature, are the round spheres. In his turn, Alexandrov, see \cite{Alex}, proved that the only embedded hypersurfaces of $\R^n$, compact, without boundary, with constant mean curvature, are the round spheres. But the theorems similar to the Hopf or the Alexandrov ones are not true for self-shrinkers. We know some examples of self-shrinkers, homeomorphic to the sphere, which are not the round sphere, and examples of compact, without boundary, embedded torus which are self-shrinkers, see for both cases, the examples of Drugan, see \cite{D}, and of Drugan and Kleene, see \cite{D-K}. 

In this paper, we prove that a two-dimensional self-shrinker, homeomorphic to the sphere, immersed in the three dimensional Euclidean space $\R^3$ is a round sphere, provided its mean curvature and the norm of  its position vector have an upper bound in terms of the norm of its traceless second fundamental form.

The proof of our results is inspired by the Hopf's work. Since this might be the first paper to apply the Hopf's work to self-shrinkers, let us mention briefly his proofs.  Using his quadratic differential he gave two proofs for his theorem.

In the first proof, one considers the second fundamental form $II$ in isothermal parameters and takes the $(2,0)$-component of $II$, i.e., $II^{(2,0)}=(1/2)P dz^2.$ It can be shown that the complex function $P$ is holomorphic if and only if $H$ is constant and that the zeroes of $P$ are the umbilical points of $\Sigma$. It is also seen that the quadratic form $II^{(2,0)}$ does not depend on the parameter $z$; hence, it is globally defined on $\Sigma$. It is a known theorem on Riemann surfaces that if the genus $g$ of $\Sigma$ is zero, any holomorphic quadratic form vanishes identically. Then $P=0$, i.e., all points of $\Sigma$ are umbilics, and hence $\Sigma$ is a standard sphere. 

His second proof is based on the lines of curvature. The quadratic equation $\mbox{Im}(Pdz^2)=0$ determines two fields of directions (the principal directions), whose singularities are the zeroes of $P$. Since $P$ is holomorphic, if $z_0$ is a zero of $P$, either $P=0$ in a neighborhood $V$ of $z_0$ or 
\begin{equation}\label{eq-001}
P(z)=(z-z_0)^kh_k(z), \ z\in V, \ k\geq 1,
\end{equation}
where $h_k$ is a function of $z$ with  $h_k(z_0)\neq 0,$ see for example \cite{Rudin}, p. 208-209. It follows that $z_0$ is an isolated singularity of the field of directions and its index is $-k/2,$ and hence, negative. Thus, either $II^{(2,0)}=0$ on $\Sigma,$ and we have a standard sphere, or all singularities are isolated and have negative index. Since $g = 0$, by the Poincar\'e index theorem, the sum of the indices of all singularities for any field of directions is two (hence positive). This is a contradiction, so $II^{(2,0)}=0$ on $\Sigma$. Notice that, in the second proof, the fact that $P$ is holomorphic is only used to show that the index of an isolated singularity of the field of directions is negative and that either $P=0$ or the zeroes of $P$ are isolated. 

In our first result, we will use a  weak holomorphy to obtain the same conclusion (\ref{eq-001}). This will be crucial to prove our classification theorems since the Hopf quadratic differential is not necessarily holomorphic for self-shrinkers. The existence of a weak notion of holomorphy to conclude (\ref{eq-001}) was noticed first, as we know, by Carleman in 1933, see \cite{Carleman}, and was used later by Hartman and Wintner \cite{H-W} and \cite{H-W-2}, Chern \cite{chern}, Eschenburg and Tribuzy \cite{E-T-1} and \cite{E-T}, and Alencar-do Carmo-Tribuzy \cite{A-dC-T}. We refer to Section \ref{section-complex} for more  history.

%\begin{theorem}\label{main}
%Let $h:U\subset\CC\rightarrow\CC$ be a complex function defined in an open set $U$ of the complex plane and $z=z_0\in U$ be a zero of $h.$  Assume either one of the following assertions holds:
%\begin{itemize}
%\item[i)] There exist $\vp\in L^p_{loc}(U),\ p>2,$ a non-negative real function such that 
%\begin{equation}\label{cauchy}
%\left|\frac{\partial h}{\partial\bar{z}}\right|\leq \vp(z)G(|h(z)|),
%\end{equation}
%where $G:[0,\infty)\rightarrow[0,\infty)$ a is locally integrable function such that $\limsup_{t\rightarrow0^+}G(t)/t<\infty$.
%
%\item[ii)] There exists a non-negative real function $\vp,$ defined on $U,$ satisfying $\vp(z)/(z-z_0)^r \in L^p_{loc}(U),$ $p>2,$ for every $r>0,$ such that 
%\begin{equation}\label{cauchy-000}
%\left|\frac{\partial h}{\partial\bar{z}}\right|\leq \vp(z)|h(z)|^\alpha, \ \alpha\in(0,1).
%\end{equation}
%\end{itemize}

\begin{theorem}\label{main}
Let $h:U\subset\CC\rightarrow\CC$ be a complex function defined in an open set $U$ of the complex plane and $z=z_0\in U$ be a zero of $h.$ If there exists $\vp\in L^p_{loc}(U),\ p>2,$ a non-negative real function such that 
\begin{equation}\label{cauchy}
\left|\frac{\partial h}{\partial\bar{z}}\right|\leq \vp(z)G(|h(z)|),
\end{equation}
where $G:[0,\infty)\rightarrow[0,\infty)$  is a locally integrable function such that $\limsup_{t\rightarrow0^+}G(t)/t<\infty$, then either $h=0$ in a neighborhood $V\subset U$ of $z_0,$ or
\[
h(z)=(z-z_0)^k h_k(z), \ z\in V, \ k\geq1,
\]
where $h_k(z)$ is a continuous function with $h_k(z_0)\neq 0.$
\end{theorem}

%777777777777777777777777777777777777777777777777777777777777

%\begin{remark}
%{\normalfont
%The there are many functions satisfying the condition $\limsup_{t\rightarrow0}G(t)/t.$ In fact, if $G$ is a continuous function, then $\limsup_{t\rightarrow0^+}G(t)/t=G'(0_+),$ if it exists. Moreover, if $G$ is any convex function with $G(0)=0$, then $G(t)/t\leq G(1)$ for small $0<t<1,$ which implies that convex function also satisfies the condition. In particular, the functions $G(t)=t^\alpha, \alpha\geq 1$ satisfy the contidion. On the other hand, are concave functions which satisfies this condition, for example $G(t)=\sin t,\ 0\leq t \leq \pi/2.$ However the class of functions $G(t)=t^\alpha, 0<\alpha<1,$ does not satisfy the condition.
%}
%\end{remark}

\begin{corollary}
Let $h:U\subset\CC\rightarrow\CC$ be a complex function defined in an open set $U$ of the complex plane. If (\ref{cauchy}) holds, then on each connected components of $U$ which contains a zero of $h$, either $h\equiv 0$ or the  zeroes of $h$ are isolated.
\end{corollary}

\begin{remark}
%{\normalfont
The case when $\vp=0$ is equivalent to that $h$ is holomorphic. The case when $G(t)=t$ and $\vp$ is continuous,  Theorem \ref{main} is the Main Lemma in \cite{A-dC-T} which implies Chern's Lemma in \cite{chern}.  Theorem \ref{main} also implies  Lemma 2.3, p. 154, of \cite{E-T}. There are many functions satisfying the condition $\limsup_{t\rightarrow0}G(t)/t<\infty.$ In fact, if $G$ is a continuous function such that $G(0)=0$, then $\limsup_{t\rightarrow0}G(t)/t=G'(0),$ if it exists. Moreover, if $G$ is any convex function with $G(0)=0$, then $G(t)/t\leq G(1)$ for small $0<t<1,$ which implies that convex functions also satisfy the condition. In particular, the functions $G(t)=t^\alpha, \alpha\geq 1,$ satisfy the condition. On the other hand, there are concave functions which satisfy this condition, for example $G(t)=\sin t,\ 0\leq t \leq \pi/2.$ %However the class of functions $G(t)=t^\alpha, 0<\alpha<1,$ does not satisfy the condition.
%}
\end{remark}

\begin{remark}
%{\normalfont
The Hopf quadratic differential has a beautiful generalisation by Abresch and Rosenberg, see \cite{AR} and \cite{AR-2}. They showed the existence of a quadratic differential which is holomorphic for constant mean curvature surfaces in the three-dimensional simply connected homogeneous spaces with four dimensional isometry group, extending the well known Hopf's theorem to these spaces. 

%The Hopf quadratic differential has a beautiful generalisation by Abresch and Rosenberg to constant mean curvature surfaces in the three-dimensional simply connected homogeneous spaces with four dimensional isometry group, see \cite{AR} and \cite{AR-2}. 

%They showed that, on constant mean curvature surfaces, the Abresch-Rosenberg quadratic differential is holomorphic on constant mean curvature surfaces.
%}
\end{remark}

Applying Theorem \ref{main}, we prove the following rigidity result for self-shrinkers:

\begin{theorem}\label{Gauss-space-2}
Let $X:\Sigma\to\R^3$ be an immersed self-shrinker homeomorphic to the sphere. If there exists a non-negative locally $L^p$ function $\vp:\Sigma\to\R,\ p>2,$ and a locally integrable function $G:[0,\infty)\rightarrow[0,\infty)$ satisfying $\limsup_{t\rightarrow 0}G(t)/t<\infty,$ such that  
\begin{equation}\label{HG-2}
(\|X\|^2-4H^2)H^2\leq \vp^2 G(\|\Phi\|)^2,
\end{equation}
then $X(\Sigma)$ is a round sphere of radius $2$ and centered at the origin.

Here  $\|\Phi\|$ denotes the matrix norm of $\Phi=A-(H/2)I,$ where $A$ is the shape operator of the second fundamental form of $X,$ $H$ is its non-normalized mean curvature, and $I$ is the identity operator of $T\Sigma.$
\end{theorem}

\begin{remark}
%{\normalfont
The hypothesis (\ref{HG-2}) of Theorem \ref{Gauss-space-2} is necessary. In fact, Drugan constructed in \cite{D} an example of an immersed rotational self-shrinker, homeomorphic to the sphere, which is not the round sphere. In section \ref{sec.ex} we prove that this example of self-shrinker does not satisfy (\ref{HG-2}).
%}
\end{remark}
\begin{remark}
The hypothesis (\ref{HG-2}) of Theorem \ref{Gauss-space-2} has also a natural geometric interpretation, as follows. We know that, for every surface of $\mathbb{R}^3$, $H^2-4K= \|\Phi\|^2\geq0,$ i.e.,
\[
K\leq \frac{1}{4}H^2,
\]
where $K$ denotes the Gaussian curvature of $\Sigma.$ We claim that hypothesis (\ref{HG-2}) gives the existence of a function $\psi:\Sigma\to\mathbb{R}$ such that
\begin{equation}\label{2}
K\leq \frac{1}{4}(1-\psi^2)H^2.
\end{equation}
Moreover, if the function $G$ satisfies $G(t)>0$ for every $t\neq 0$ and $\limsup_{t\to0}G(t)/t\neq 0,$ then this function $\psi$ can be chosen in such way that, for every $\varepsilon>0$ arbitrarily small, we have
\begin{equation}\label{3}
\psi^2<\varepsilon.
\end{equation}
In order to prove (\ref{2}), notice that, by using (\ref{HG-2}),
\[
(\|X\|^2-4H^2)H^2\leq \varphi^2 G(\|\Phi\|)^2 = \varphi^2 \frac{G(\|\Phi\|)^2}{\|\Phi\|^2}(H^2-4K),
\]
which implies, by a rearrangement of the terms,
\[
K \leq \frac{1}{4}\left[1-\frac{1}{\varphi^2}\left(\frac{\|\Phi\|}{G(\|\Phi\|)}\right)^2(\|X\|^2-4H^2)\right]H^2.
\]
Inequality (\ref{2}) follows by taking 
\begin{equation}\label{4}
\psi = \frac{1}{\varphi}\frac{\|\Phi\|}{G(\|\Phi\|)}\sqrt{\|X\|^2-4H^2}.
\end{equation}
Since the function $G$ satisfies $G(t)>0$ for every $t\neq 0$ and $\limsup_{t\to0}G(t)/t\neq 0,$ there exists
\[
M:=\sup_{\Sigma}\frac{\|\Phi\|}{G(\|\Phi\|)}\sqrt{\|X\|^2-4H^2}.
\]
Therefore, given an arbitrary $\varepsilon>0,$ by choosing $\varphi$ as a constant function large enough such that $\varphi > M\varepsilon^{-1/2},$ we obtain (\ref{3}) from (\ref{4}).
\end{remark}

Theorem \ref{Gauss-space-2} motivates us to study the zeroes of the functions $H^2$ and $\|X\|^2-4H^2$ at the zeroes of $\|\Phi\|^2$.

\begin{defn}\label{defn-1}
Let $z_0$ be a zero point of a function $\psi$. The lower order of the zero $\zeta_-^\psi(z_0)$ is defined as the biggest number $a$ such that 
\[
\liminf_{z\to z_0}\frac{|\psi(z)|}{(\dist(z,z_0))^a}> 0.
\]
The upper order of the zero $\zeta_+^\psi(z_0)$ is defined as the smallest number $a$ such that 
\[
\limsup_{z\to z_0}\frac{|\psi(z)|}{(\dist(z,z_0))^a}<+\infty.
\]
\end{defn}

As a consequence of Theorem \ref{Gauss-space-2}, we present the following result, which will be proven in the section \ref{H-f}, p. \pageref{cor-1-1}.

\begin{corollary}\label{cor-1}
Let $X:\Sigma\to\R^3$ be an immersed self-shrinker homeomorphic to the sphere. If at each umbilical points, the lower order of $\|\Phi\|^2$ minus the upper order of the function $(\|X\|^2-4H^2)H^2$ is less than $2$, then $X(\Sigma)$ is a round sphere of radius $2$ and centered at the origin.
\end{corollary}

\begin{remark}
%{\normalfont
There are many other results of rigidity of the round spheres as the only compact self-shrinkers. In dimension $n,$ Huisken, see \cite{huisken}, proved that the sphere of radius $\sqrt{2n}$ is the only compact, mean convex, self-shrinker in the Euclidean space. Colding and Minicozzi \cite{C-M} proved that the sphere of radius $\sqrt{2n}$ is also the only compact $F$-stable self-shrinker in the Euclidean space. In their turn, Kleene and Moller, see \cite{K-M}, proved that the sphere of radius $\sqrt{2n}$ is the only rotationally symmetric, embedded self-shrinker in the Euclidean space which is homeomorphic to the sphere. In \cite{Cao-Li}, Cao and Li proved that complete $n$-dimensional self-shrinkers in $\R^{n+k},\ k\geq 1,$ with polynomial volume growth, and such that $\|A\|^2\leq \frac12$ are spheres, cylinders or hyperplanes. Here $\|A\|^2$ means the squared norm of the second fundamental form of the self-shrinker in $\R^{n+k}.$ We can also cite the result of Brendle \cite{brendle} who proved that the only closed, embedded self-shrinkers in $\R^3$ with genus zero, are the round spheres. 
%}
\end{remark}

\begin{remark}
%{\normalfont
Theorem \ref{Gauss-space-2} is a consequence of the more general result Theorem \ref{Gauss-space}, p. \pageref{Gauss-space}, which holds for parallel weighted mean curvature surfaces in $\R^{2+m},\ m\geq1,$ where the weight is a radial function (i.e., which depends only on the distance of the point to the origin), see section \ref{H-f} for the precise definitions. As consequences of this theorem, we prove rigidity results in the same spirit of Theorem \ref{Gauss-space-2} for constant weighted mean curvature surfaces with the Gaussian measure, also called $\lambda$-surfaces. These surfaces, which are characterized by the equation
\[
\lambda = H + \frac{1}{2}\lan X,N\ran
\]
for each $\lambda\in\R,$ have been intensively studied in recent years, see for example, \cite{M-R}, \cite{C-O-W}, \cite{MR3803340}, \cite{C-W}, and \cite{guang}. The  simple examples are round spheres centered at origin  and all the hyperplanes.  Observe that self-shrinkers are special cases of these surfaces, by taking $\lambda=0.$
%}
\end{remark}

Here is the plan of the rest of the paper: the section \ref{section-complex} is dedicated to the proof of Theorem \ref{main}. In the section \ref{H-f} we prove the results about self-shrinkers, constant weighted mean curvature surfaces, and  $f$-minimal surfaces. We conclude the paper analyzing the umbilical points of rotational self-shrinkers, especially the Drugan's example, to obtain counterexamples to the conclusion of Theorem \ref{Gauss-space-2} when the hypothesis (\ref{HG-2}) is removed.
\bigskip

{\bf  Acknowledgments.} The authors dedicate this article in memory to their professor and friend Manfredo do Carmo for his remarkable contributions to differential geometry and for his essential influence on their academic and personal experiences.

\section{Proof of Theorem \ref{main}}\label{section-complex}

%Since $P$ is holomorphic, if $z_0$ is a zero of $P$, either $P=0$ in a neighborhood $V$ of $z_0$ or 
%\begin{equation}\label{eq-001}
%P(z)=(z-z_0)^kh_k(z), \ z\in V, \ k\geq 1,
%\end{equation}
%where $h_k$ is a function of $z$ with  $h_k(z_0)\neq 0,$ see for example \cite{Rudin}, p. 208-209.

In this section we prove Theorem \ref{main}. We start with   the history line of the weak notion of holomorphy. 

A well known property of holomorphic functions establishes that if $z_0$ is a zero of a holomorphic function $h(z),$ then $h=0$ in a neighborhood of $z_0$ or there exists $k>0$ such that \[h(z)=(z-z_0)^kh_k(z),\] for some function $h_k$ such that $h_k(z_0)\neq 0,$ see for example \cite{Rudin}, p. 208-209. This number $k$ is called the order of the zero. In particular, if $h$ is not identically zero in a neighborhood of $z_0,$ then $z_0$ is isolated.

In 1933, Carleman \cite{Carleman} was the first to observe that this property holds for non-analytic smooth functions which satisfies some first order partial differential equation. In fact, he proved that a solution $h:U\subset \CC \rightarrow\CC$ of 
\[
\frac{\partial h}{\partial\bar{z}} = ah+b\bar{h},
\]
does not admits a zero of infinite order except if $h=0.$ Here bars mean complex conjugate and $a,b$ are continuous complex functions. Notice that, if $a=b=0,$ then $h$ is holomorphic. Using these ideas, Hartman and Wintner, see \cite{H-W} and \cite{H-W-2}, and Chern, see \cite{chern}, proved their well known results on the classification of special Weingarten surfaces.

The proof of Theorem \ref{main} follows the same lines. In order to simplify the notations, we will assume $z_0=0$ in the lemmas below and in the proof of the theorem. Denote also by $D_c(\tilde{z})\subset\CC$ the disc of radius $c>0$ and center $\tilde{z}\in\CC.$ In the the proof of Theorem \ref{main}, we will need the following three technical lemmas.

\begin{lemma}\label{lemma-tech-1}
Let $h:U\subset\CC\rightarrow\CC$ be a locally integrable complex function defined in a open set $U$ of the complex plane. Assume there exists $M:=\sup_{D_R(0)}|h(z)/z^{k-1}|$ for some $k\geq 1$ and for some $R>0.$ Then, for every $q\in (1,2)$ and for every $\xi\in\CC\backslash\{0\}$ we have
\[
\int_{D_R(0)}\left|\frac{h(z)}{z^k(z-\xi)}\right|^q|dz\wedge d\bar{z}|\leq M^q K_q|\xi|^{2-2q},
\]
where
\[
K_q:=\int_{\CC}\frac{|dw\wedge d\bar{w}|}{|w(w-1)|^q}<\infty.
\]
In particular, if $\lim_{z\rightarrow 0} h(z)/z^{k-1}=0,$ the same conclusion holds for a sufficiently small $R>0.$
\end{lemma}

\begin{proof}

By taking $z=\xi w$ and using the hypothesis, we have
\[
\begin{aligned}
\int_{D_R(0)}\frac{|h(z)|^q}{|z^k(z-\xi)|^q}|dz\wedge d\bar{z}|&\leq M^q \int_{D_R(0)}\frac{1}{|z(z-\xi)|^q}|dz\wedge d\bar{z}|\\
&=M^q|\xi|^{2-2q}\int_{B_{R/|\xi|}(0)}\frac{1}{|w(w-1)|^q}|dw\wedge d\bar{w}|\\
&\leq M^q|\xi|^{2-2q}\int_{\CC}\frac{1}{|w(w-1)|^q}|dw\wedge d\bar{w}|.\\
\end{aligned}
\]
On the other hand, see Figure \ref{fig-1}, by using polar coordinates $w=\rho e^{i\theta},$
\begin{equation}\label{ineq-Kp}
\begin{aligned}
\int_{\CC}\frac{1}{|w(w-1)|^q}|dw\wedge d\bar{w}|&=\int_{\CC\backslash D_2(0)}\frac{1}{|w(w-1)|^q}|dw\wedge d\bar{w}|\\
&\qquad + \int_{D_2(0)\backslash (D_\ve(0)\cup D_\ve(1))}\frac{1}{|w(w-1)|^q}|dw\wedge d\bar{w}|\\
&\qquad + \int_{D_\ve(0)}\frac{1}{|w(w-1)|^q}|dw\wedge d\bar{w}| + \int_{D_\ve(1)}\frac{1}{|w(w-1)|^q}|dw\wedge d\bar{w}|\\
&=\int_2^\infty\int_0^{2\pi}\frac{d\rho d\theta}{\rho^{q-1}|\rho e^{i\theta}-1|^q}\\
&\qquad + \int_{D_2(0)\backslash (D_\ve(0)\cup D_\ve(1))}\frac{1}{|w(w-1)|^q}|dw\wedge d\bar{w}|\\
&\qquad + \int_0^\ve\int_0^{2\pi}\frac{d\rho d\theta}{\rho^{q-1}|\rho e^{i\theta}-1|^q} + \int_0^\ve\int_0^{2\pi}\frac{d\rho d\theta}{|\rho e^{i\theta}+1|^q\rho^{q-1}},\\
\end{aligned}
\end{equation}
where, in the last of the four integrals of (\ref{ineq-Kp}), we used $w=1+\rho e^{i\theta}.$
\begin{figure}[ht]
\centering
\includegraphics[scale=0.45]{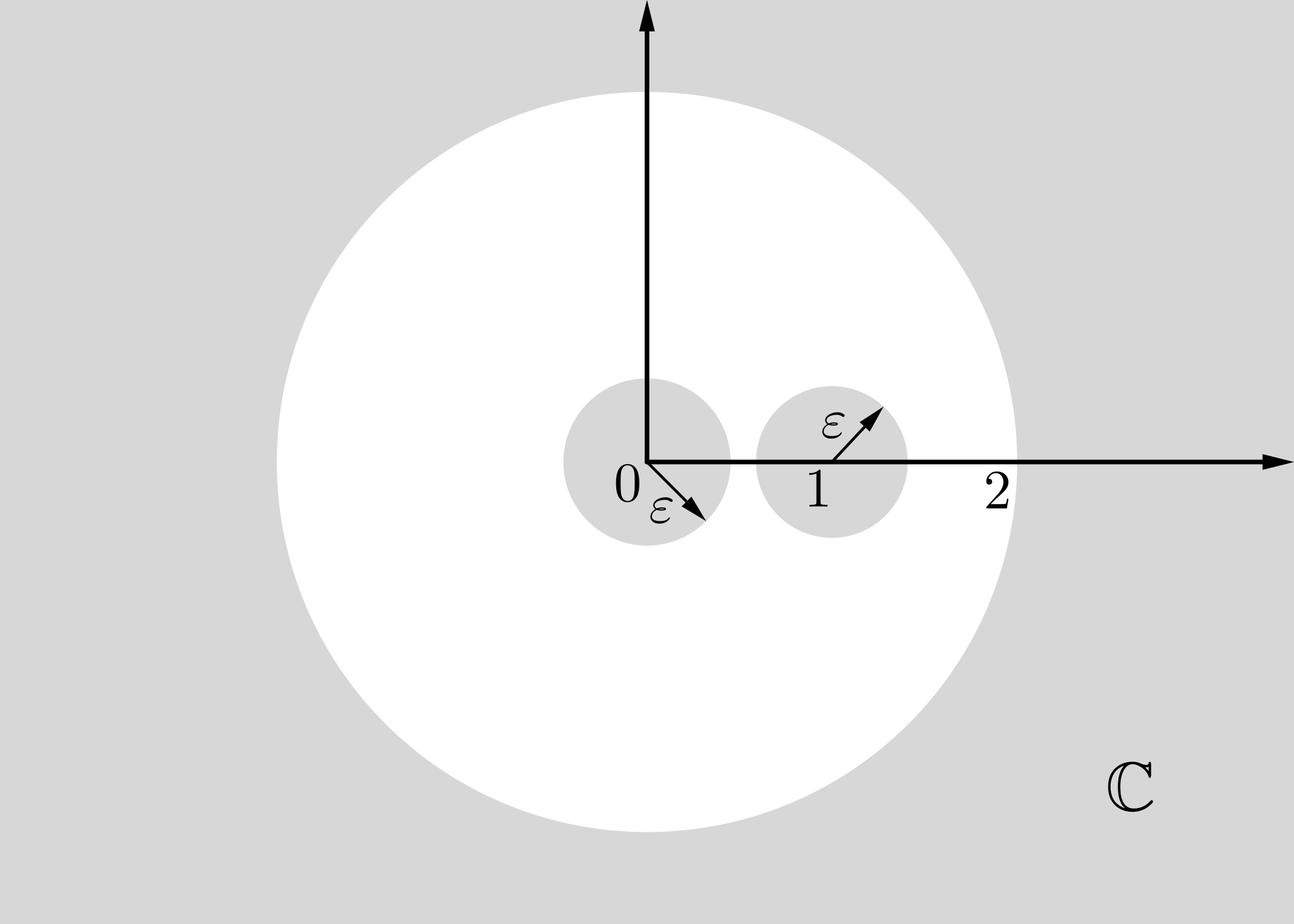}
\caption{Representation of the domains in (\ref{ineq-Kp})}
\label{fig-1}
\end{figure}
Since $|x\pm y|\geq ||x|-|y||$ and $q\in(1,2),$ we have
\[
\begin{aligned}
\int_2^\infty\int_0^{2\pi}\frac{d\rho d\theta}{\rho^{q-1}|\rho e^{i\theta}-1|^q}&\leq \frac{1}{2^{q-1}}\int_2^\infty\int_0^{2\pi}\frac{d\rho d\theta}{(\rho-1)^q}=\frac{\pi}{2^{q-2}}\int_1^\infty \frac{d\rho}{\rho^q}<\infty,\\
\int_0^\ve\int_0^{2\pi}\frac{d\rho d\theta}{\rho^{q-1}|\rho e^{i\theta}-1|^q}&\leq\int_0^\ve\int_0^{2\pi}\frac{d\rho d\theta}{\rho^{q-1}(1-\rho)^q}\leq \frac{2\pi}{(1-\ve)^q}\int_0^\ve \frac{d\rho}{\rho^{q-1}}<\infty,\\
\end{aligned}
\]
and
\[
\int_0^\ve\int_0^{2\pi}\frac{d\rho d\theta}{|\rho e^{i\theta}+1|^q\rho^{q-1}}\leq \int_0^\ve\int_0^{2\pi}\frac{d\rho d\theta}{(1-\rho)^q\rho^{q-1}}\leq \frac{2\pi}{(1-\ve)^q}\int_0^\ve \frac{d\rho}{\rho^{q-1}}<\infty.\\
\]
Therefore,
\[
K_q := \int_{\CC}\frac{|dw\wedge d\bar{w}|}{|w(w-1)|^q}<\infty
\]
and thus
\[
\int_{D_R(0)}\frac{|h(z)|^q}{|z^k(z-\xi)|^q}|dz\wedge d\bar{z}|\leq M^q K_q |\xi|^{2-2q} <\infty
\]
for every fixed $\xi\in\CC\backslash\{0\}.$
\end{proof}
\begin{lemma}[Cauchy-Pompeiu formula, adapted]
Let $h:D_R(0)\subset\CC\rightarrow\CC$ be a complex function such that $\partial h /\partial \bar{z}$ exists and it is locally integrable. If $\lim_{z\rightarrow0}h(z)/z^{k-1}=0,$ then
\begin{equation}\label{l-2}
2\pi i h(\xi)\xi^{-k} =\int_{\partial D_R(0)}\frac{h(z)}{z^k(z-\xi)}dz + \int_{D_R(0)}\frac{1}{z^k(z-\xi)}\frac{\partial h}{\partial\bar{z}}dz\wedge d\bar{z},
\end{equation}
where $\xi\in\CC\backslash\{0\}$ and $\partial D_R(0)=\{z\in\CC; |z|=R\}$ denotes the boundary of $D_R(0).$
\end{lemma}

\begin{proof}

Define the $1$-form
\[
\phi(z)=\frac{h(z)}{z^k(z-\xi)}dz.
\]
Let $W=D_R(0)\backslash(D_a(0)\cup D_a(\xi))$ for some $a>0$ sufficiently small, see Figure \ref{fig-2}. Since $1/z^k(z-\xi)$ is holomorphic in $W$, then
\[
d\phi = \frac{\partial\phi}{\partial\bar{z}}d\bar{z}\wedge dz = -\frac{1}{z^k(z-\xi)}\frac{\partial h}{\partial\bar{z}} dz\wedge d\bar{z}.
\]
By using Stokes' theorem, we have
\begin{equation}\label{l-1}
\int_W d\phi = \int_{\partial W} \phi = \int_{\partial D_R(0)}\phi - \int_{\partial D_a(0)}\phi - \int_{\partial D_a(\xi)}\phi.
\end{equation}
\begin{figure}[ht]
\centering
\includegraphics[scale=0.45]{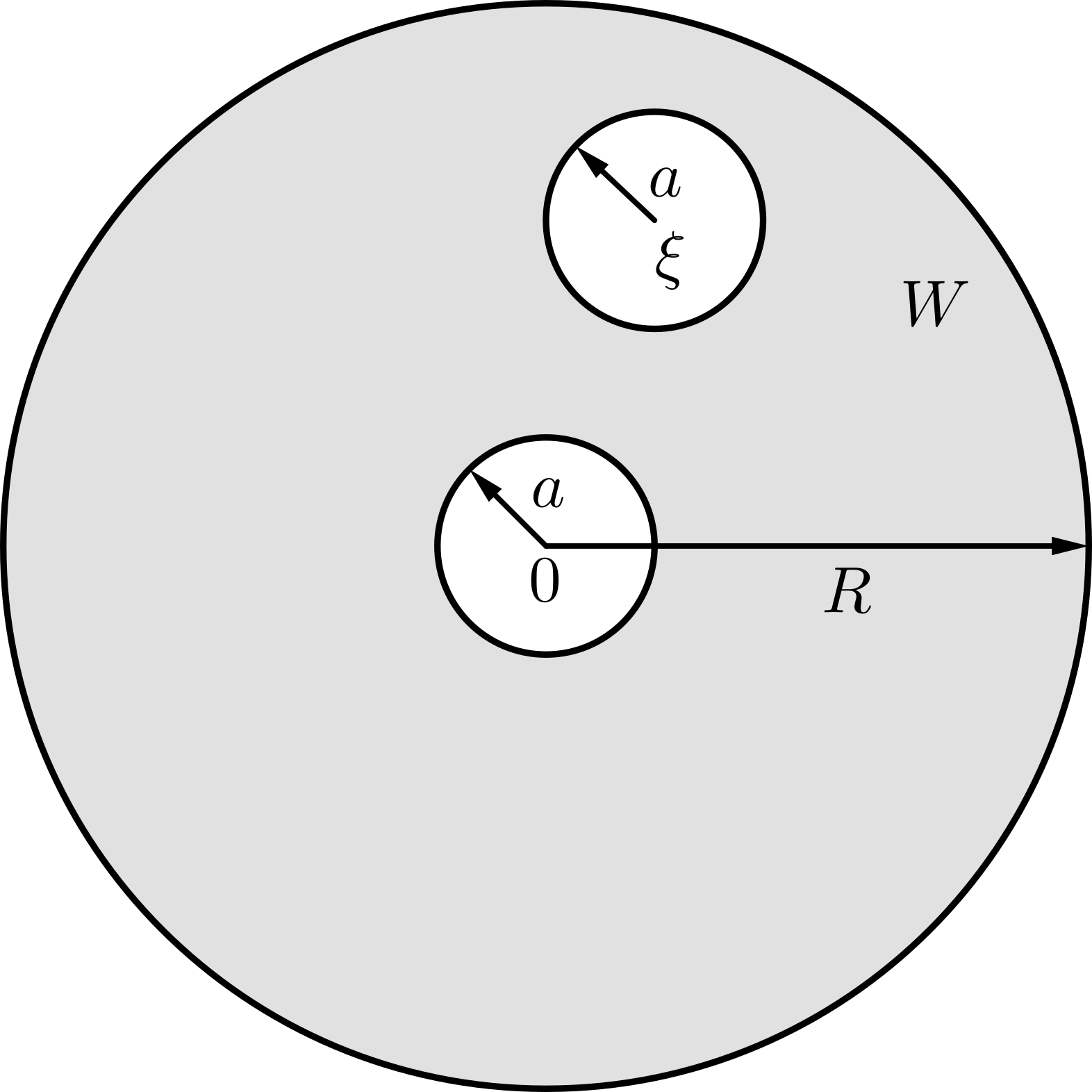}
\caption{Representation of $W$}
\label{fig-2}
\end{figure}

Let us calculate the integrals of the right hand side of (\ref{l-1}) and take $a\rightarrow0.$ Making $z=ae^{i\theta}$ in $\partial D_a(0),$ we obtain
\[
\lim_{a\rightarrow0}\int_{\partial D_a(0)}\phi = \lim_{a\rightarrow0}\int_0^{2\pi} \frac{h(ae^{i\theta})iae^{i\theta}}{a^ke^{ik\theta}(ae^{i\theta}-\xi)}d\theta = i\lim_{a\rightarrow0}\int_0^{2\pi} \frac{h(ae^{i\theta})}{(ae^{i\theta})^{k-1}(ae^{i\theta}-\xi)}d\theta=0
\]
since $\lim_{a\rightarrow0}\frac{h(ae^{i\theta})}{(ae^{i\theta})^{k-1}}=0$ by hypothesis. On the other hand, making $z=\xi+ae^{i\theta}$ in $\partial B_a(\xi),$ we have
\[
\lim_{a\rightarrow0}\int_{\partial D_a(\xi)}\phi = \lim_{a\rightarrow0}\int_0^{2\pi} \frac{h(\xi+ae^{i\theta})iae^{i\theta}}{(\xi+ae^{i\theta})^k ae^{i\theta}}d\theta = i\lim_{a\rightarrow0}\int_0^{2\pi} \frac{h(\xi+ae^{i\theta})}{(\xi+ae^{i\theta})^k}d\theta =2\pi i h(\xi)\xi^{-k}.
\]
Thus, taking $a\rightarrow0$ in (\ref{l-1}) gives
\[
-\int_{D_R(0)}\frac{1}{z^k(z-\xi)}\frac{\partial h}{\partial\bar{z}}dz\wedge d\bar{z} = \int_{\partial D_R(0)}\frac{h(z)}{z^k(z-\xi)}dz - 2\pi i h(\xi)\xi^{-k}.
\]
\end{proof}
\begin{lemma}\label{lemma-fubini}
Let $h:U\subset\CC\rightarrow\CC$ be a locally integrable complex function, where $U$ is an open neighborhood of $0$. Let $z_0\in\CC\backslash\{0\}$ and, for $\ve>0,$ $D_\ve = D_R(0)\backslash(D_\ve(0)\cup D_\ve(z_0)).$ If there exists $M:=\sup_{D_R(0)}|h(z)/z^{k-1}|$ for some $k\geq 1$ and for some $R>0$ then, for every $q\in (1,2)$ and for every $\xi\in\CC,$
\[
\begin{aligned}
\int_{D_\ve}\left[\int_{D_R(0)}\left|\frac{h(z)}{z^k(z-\xi)(\xi-z_0)}\right|^q du\wedge dv\right]& dx\wedge dy\\
&\leq \frac{2^{q+1}\pi(2R)^{2-q}}{2-q}\int_{D_R(0)}\left|\frac{h(z)}{z^k(z-z_0)}\right|^q du\wedge dv\\
\end{aligned}
\]
and
\[
\begin{aligned}
\int_{D_\ve}\left[\int_{\partial D_R(0)}\left|\frac{h(z)}{z^k(z-z_0)}\right|^q|dz|\right]&\frac{dx\wedge dy}{|\xi-z_0|^q}\\
&\leq \frac{2^{q+1}\pi(2R)^{2-q}}{2-q}\int_{\partial D_R(0)}\left|\frac{h(z)}{z^k(z-z_0)}\right|^q|dz|.\\
\end{aligned}
\]
In particular, the same conclusion holds if $\lim_{z\rightarrow0} h(z)/z^{k-1}=0.$
\end{lemma}
\begin{proof}
Since the convexity of the function $g(x)=x^q,$ $q\in(1,2),$ gives 
\[
\left(\frac{A+B}{2}\right)^q\leq \frac{A^q+B^q}{2},\ \mbox{for}\ A,B>0,
\]
we have
\[
\frac{1}{(z-\xi)(\xi-z_0)}=\frac{1}{z-z_0}\left(\frac{1}{z-\xi}+\frac{1}{\xi-z_0}\right).
\]
This implies
\begin{equation}\label{qq-1}
\begin{aligned}
\frac{1}{|z-\xi|^q|\xi-z_0|^q}&\leq \left[\frac{1}{|z-z_0|}\left(\frac{1}{|z-\xi|} + \frac{1}{|\xi-z_0|}\right)\right]^q\\
&\leq \frac{2^{q-1}}{|z-z_0|^q|z-\xi|^q} + \frac{2^{q-1}}{|z-z_0|^q|\xi-z_0|^q}. 
\end{aligned}
\end{equation}
Also, by taking $\xi=z_0+\rho e^{i\theta},$
\begin{equation}\label{qq-2}
\int_{D_R(0)}\frac{dx\wedge dy}{|\xi - z_0|^q} \leq \int_{D_{2R}(z_0)}\frac{dx\wedge dy}{|\xi - z_0|^q}=\int_0^{2R}\int_0^{2\pi}\frac{d\theta d\rho}{\rho^{q-1}} = \frac{2\pi(2R)^{2-q}}{2-q}<\infty.
\end{equation}
Since, by Lemma \ref{lemma-tech-1},
\[
\int_{D_R(0)}\left|\frac{h(z)}{z^k(z-\xi)}\right|^qdu \wedge dv < \infty,
\]
for each fixed $\xi\neq 0,$ by using (\ref{qq-1}), Fubini's theorem over $D_\ve = D_R(0)\backslash(D_\ve(z_0)\cup D_\ve(0))$, and (\ref{qq-2}), we obtain
\[
\begin{aligned}
\int_{D_\ve}&\left[\int_{D_R(0)}\left|\frac{h(z)}{z^k(z-\xi)(\xi-z_0)}\right|^q du\wedge dv\right] dx\wedge dy\\
&\qquad\qquad\qquad \leq 2^{q-1}\int_{D_\ve}\left[\int_{D_R(0)}\left|\frac{h(z)}{z^k(z-\xi)(z-z_0)}\right|^q du\wedge dv\right] dx\wedge dy\\
&\qquad\qquad\qquad\qquad +2^{q-1}\int_{D_\ve}\left[\int_{D_R(0)}\left|\frac{h(z)}{z^k(z-z_0)(\xi-z_0)}\right|^q du\wedge dv\right] dx\wedge dy\\
&\qquad\qquad\qquad=2^{q-1}\int_{D_R(0)}\left|\frac{h(z)}{z^k(z-z_0)}\right|^q\left[\int_{D_\ve}\frac{dx\wedge dy}{|\xi-z|^q}\right] du\wedge dv\\
&\qquad\qquad\qquad\qquad + 2^{q-1} \int_{D_R(0)}\left|\frac{h(z)}{z^k(z-z_0)}\right|^q\left[\int_{D_\ve}\frac{dx\wedge dy}{|\xi - z_0|^q}\right]du\wedge dv\\
&\qquad\qquad\qquad\leq \frac{2^{q+1}\pi(2R)^{2-q}}{2-q}\int_{D_R(0)}\left|\frac{h(z)}{z^k(z-z_0)}\right|^q du\wedge dv.
\end{aligned}
\]
Analogously,
\[
\int_{D_\ve}\left[\int_{\partial D_R(0)}\left|\frac{h(z)}{z^k(z-z_0)}\right|^q|dz|\right]\frac{dx\wedge dy}{|\xi-z_0|^q}\leq \frac{2^{q+1}\pi(2R)^{2-q}}{2-q}\int_{\partial D_R(0)}\left|\frac{h(z)}{z^k(z-z_0)}\right|^q|dz|.
\]
\end{proof}

Now we are ready to prove Theorem \ref{main}.

\begin{proof}[Proof of Theorem \ref{main}.] The proof will be divided in four steps.

\emph{Step 1. If $\lim_{z\rightarrow 0} h(z)/z^{k-1}=0$ for some $k\geq 1,$ then $h(z)/z^k$ is bounded in $D_R(0)$ for $R>0$ fixed, but sufficiently small.}

By using the Cauchy-Pompeiu formula (\ref{l-2}), p. \pageref{l-2}, the hypothesis, and the H\"older inequality, we have
\begin{equation}\label{eq-q1-1}
\begin{aligned}
2\pi\left|\frac{h(\xi)}{\xi^k}\right|&\leq \int_{\partial D_R(0)}\frac{|h(z)||dz|}{|z^k(z-\xi)|} + \int_{D_R(0)}\frac{1}{|z^k(z-\xi)|}\left|\frac{\partial h}{\partial\bar{z}}\right||dz\wedge d\bar{z}|\\
&\leq\int_{\partial D_R(0)}\frac{|h(z)||dz|}{|z^k(z-\xi)|} +  \int_{D_R(0)}\frac{\vp(z)G(|h(z)|)}{|z^k(z-\xi)|}|dz\wedge d\bar{z}|\\
&\leq\int_{\partial D_R(0)}\frac{|h(z)||dz|}{|z^k(z-\xi)|} +  \sup_{D_R(0)}\left\{\frac{G(|h(z)|)}{|h(z)|}\right\}\int_{D_R(0)}\frac{\vp(z)|h(z)|}{|z^k(z-\xi)|}|dz\wedge d\bar{z}|\\
&\leq(2\pi R)^{1/p}\left[\int_{\partial D_R(0)}\left|\frac{h(z)}{z^k(z-\xi)}\right|^q|dz|\right]^{1/q}\\
&\qquad +  M_R\|\vp\|_{p,R}\left[\int_{D_R(0)}\left|\frac{h(z)}{z^k(z-\xi)}\right|^q|dz\wedge d\bar{z}|\right]^{1/q},\\
\end{aligned}
\end{equation}
where 
\begin{equation}\label{MR}
M_R:=\sup_{D_R(0)}\left\{\frac{G(|h(z)|)}{|h(z)|}\right\}<\infty,
\end{equation}
by hypothesis, and
\[
\|\vp\|_{p,R}=\left[\int_{D_R(0)}\vp(z)^p|dz\wedge d\bar{z}|\right]^{1/p}
\]
is the $L^p$ norm of $\vp$ in $D_R(0)$. Notice that the second integral of the right hand side of (\ref{eq-q1-1}) is bounded for every fixed $\xi\neq 0$ by Lemma \ref{lemma-tech-1}, p. \pageref{lemma-tech-1}. By using
\[
\left(\frac{A+B}{2}\right)^q\leq \frac{A^q+B^q}{2},\ \mbox{for}\ A,B>0,\ q\in(1,2),
\] 
we obtain
\begin{equation}\label{int-111}
\begin{aligned}
(2\pi)^q\left|\frac{h(\xi)}{\xi^k}\right|^q\!\!&\leq\!\! \left\{(2\pi R)^{1/p}\left[\int_{\partial D_R(0)}\left|\frac{h(z)}{z^k(z-\xi)}\right|^q\!\!|dz|\right]^{1/q}\right.\\
&\left.\qquad\qquad\qquad + M_R\|\vp\|_{p,R}\left[\int_{D_R(0)}\left|\frac{h(z)}{z^k(z-\xi)}\right|^q\!\!|dz\wedge d\bar{z}|\right]^{1/q}\right\}^q\\
&\leq 2^{q-1}(2\pi R)^{q-1}\int_{\partial D_R(0)}\left|\frac{h(z)}{z^k(z-\xi)}\right|^q\!\!|dz|\\
&\qquad\qquad\qquad + 2^{q-1}M_R^q\|\vp\|_{p,R}^q \int_{D_R(0)}\left|\frac{h(z)}{z^k(z-\xi)}\right|^q\!\!|dz\wedge d\bar{z}|.\\
\end{aligned}
\end{equation}
Multiplying inequality (\ref{int-111}) by $|\xi-z_0|^{-q}, \ z_0\in D_R(0)\backslash\{0\},$ and integrating on $D_\ve = D_R(0)\backslash(D_\ve(z_0)\cup D_\ve(0))$ gives
\[
\begin{aligned}
\int_{D_\ve}\left|\frac{h(\xi)}{\xi^k(\xi-z_0)}\right|^q dx\wedge dy &\leq \frac{R^{q-1}}{2^{2-q}\pi}\int_{D_\ve}\left[\int_{\partial D_R(0)}\left|\frac{h(z)}{z^k(z-\xi)}\right|^q\!\!|dz|\right]\frac{dx\wedge dy}{|\xi-z_0|^q}\\
&\qquad + \left(\frac{M_R\|\vp\|_{p,R}}{\pi}\right)^q\int_{D_\ve}\left[\int_{D_R(0)}\left|\frac{h(z)}{z^k(z-\xi)}\right|^q\!\! du\wedge dv\right]\frac{dx\wedge dy}{|\xi-z_0|^q},
\end{aligned}
\]
where $z=u+iv,$ $\xi=x+iy,$ and $|dz\wedge d\bar{z}|=2du\wedge dv.$ By using Lemma \ref{lemma-fubini}, p. \pageref{lemma-fubini}, we obtain 
\[
\begin{aligned}
\int_{D_\ve}\left|\frac{h(\xi)}{\xi^k(\xi-z_0)}\right|^q dx\wedge dy& \leq \frac{R^{q-1}}{2^{2-q}\pi}\cdot\frac{8\pi R^{2-q}}{2-q}\int_{\partial D_R(0)}\left|\frac{h(z)}{z^k(z-z_0)}\right|^q|dz|\\
&\qquad + \frac{M_R^q\|\vp\|_{p,R}^q}{\pi^q}\cdot\frac{8\pi R^{2-q}}{2-q}\int_{D_R(0)}\left|\frac{h(z)}{z^k(z-z_0)}\right|^q du\wedge dv,
\end{aligned}
\]
and thus the left hand side is bounded for every $\ve>0.$ Taking $\ve\rightarrow0,$ we have 
\begin{equation}\label{eq-q2}
\left(1-\frac{8M_R^q\|\vp\|_{p,R}^qR^{2-q}}{\pi^{q-1}(2-q)}\right)\int_{D_R(0)}\left|\frac{h(z)}{z^k(z-z_0)}\right|^q du\wedge dv \leq \frac{2^{1+q}R}{2-q}\int_{\partial D_R(0)}\left|\frac{h(z)}{z^k(z-z_0)}\right|^q|dz|.
\end{equation}
Since $2-q>0$ and $M_R\|\vp\|_{p,R}$ decreases as $R\rightarrow0,$ taking $R>0$ sufficiently small, we get
\[
1-\frac{8M_R\|\vp\|_{p,R}^qR^{2-q}}{\pi^{q-1}(2-q)}>0.
\]
By replacing (\ref{eq-q2}) in (\ref{eq-q1-1}), we obtain
\begin{equation}\label{eq-q3}
2\pi\left|\frac{h(\xi)}{\xi^k}\right|\!\leq\!\left(\!(2\pi R)^{1/p}\! +\!\left(\frac{2^{1+q}\pi^{q-1}M_R^q\|\vp\|_{p,R}^qR}{\pi^{q-1}(2-q)-8M_R^q\|\vp\|_{p,R}^qR^{2-q}}\right)^{1/q}\!\right)\left[\int_{\partial D_R(0)}\left|\frac{h(z)}{z^k(z-\xi)}\right|^q|dz|\right]^{1/q}.
\end{equation}
Since
\[
\begin{aligned}
\int_{\partial D_R(0)}\left|\frac{h(z)}{z^k(z-\xi)}\right|^q|dz| &= \int_0^{2\pi}\frac{|h(Re^{i\theta})|^q}{R^{kq-1}|Re^{i\theta}-\xi|^q}d\theta\\
&\leq \frac{1}{R^{kq-1}(R-|\xi|)^q}\int_0^{2\pi}|h(Re^{i\theta})|^q d\theta,
\end{aligned}
\]
we conclude that $h(\xi)\xi^{-k}$ is bounded for $\xi$ near zero.

\emph{Step 2. There exists $\lim_{z\rightarrow 0}h(z)z^{-k}.$} 

Define $T_1:D_R(0)\rightarrow\R$ by
\[
T_1(\xi)=\int_{D_R(0)}\frac{1}{z^k(z-\xi)}\frac{\partial h}{\partial\bar{z}}dz\wedge d\bar{z}
\]
and $T_2:D_R(0)\rightarrow\R$ by
\[
T_2(\xi)=\int_{\partial D_R(0)}\frac{h(z)}{z^k(z-\xi)}dz.
\]
Thus, by using the Cauchy-Pompeiu formula (\ref{l-2}), p. \pageref{l-2}, to prove that $\lim_{\xi\rightarrow 0}h(\xi)\xi^{-k}$ exists, one only need to prove both $\lim_{\xi\rightarrow 0}T_1(\xi)$ and $\lim_{\xi\rightarrow 0}T_2(\xi)$ exist. In fact, by the step 1, there exists $N=\sup_{D_R(0)}|h(z)z^{-k}|.$  Therefore, using the hypothesis and (\ref{MR}), p. \pageref{MR},
\[
\begin{aligned}
|T_1(\xi_2)-T_1(\xi_1)| &= \left|\int_{D_R(0)}\frac{\partial h}{\partial\bar{z}}\frac{1}{z^k}\left(\frac{1}{z-\xi_2}-\frac{1}{z-\xi_1}\right)dz\wedge d\bar{z}\right|\\
&\leq \int_{D_R(0)}\left|\frac{\partial h}{\partial\bar{z}}\right|\frac{1}{|z|^k}\left|\frac{1}{z-\xi_2}-\frac{1}{z-\xi_1}\right||dz\wedge d\bar{z}|\\
&\leq \int_{D_R(0)}\vp(z)\frac{G(|h(z)|)}{|h(z)|}\left|\frac{h(z)}{z^k}\right|\left|\frac{\xi_2-\xi_1}{(z-\xi_2)(z-\xi_1)}\right||dz\wedge d\bar{z}|\\
&\leq  M_RN|\xi_2-\xi_1|\int_{D_R(0)}\vp(z)\left|\frac{1}{(z-\xi_2)(z-\xi_1)}\right||dz\wedge d\bar{z}|\\
&\leq M_RN|\xi_2-\xi_1|\|\vp\|_{p,R}\left[\int_{D_R(0)}\frac{1}{|z-\xi_2|^q|z-\xi_1|^q}|dz\wedge d\bar{z}|\right]^{1/q}.
\end{aligned}
\]
Considering $z-\xi_1=(\xi_2-\xi_1)w$ we have
\begin{equation}\label{ineq.t1}
\begin{aligned}
|T_1(\xi_2)-T_1(\xi_1)| &\leq M_RN\|\vp\|_{p,R}|\xi_2-\xi_1|^{2/q-1}\left[\int_{D_{R/|\xi_2-\xi_1|}(\xi_1)}\left|\frac{1}{w(w-1)}\right|^q|dw\wedge d\bar{w}|\right]^{1/q}\\
&\leq M_RN\|\vp\|_{p,R}|\xi_2-\xi_1|^{1-2/p}\left[\int_{\CC}\left|\frac{1}{w(w-1)}\right|^q|dw\wedge d\bar{w}|\right]^{1/q}.\\
\end{aligned}
\end{equation}
Since the last integral in (\ref{ineq.t1}) is finite by the Lemma \ref{lemma-tech-1}, p. \pageref{lemma-tech-1}, we obtain
\[
|T_1(\xi_2)-T_1(\xi_1)| \leq (M_RN\|\vp\|_{p,R}K_q^{1/q})|\xi_2-\xi_1|^{1-2/p}.
\]
Thus, by taking Cauchy sequences and using that $p>2$, there exists $\lim_{\xi\rightarrow0}T_1(\xi).$ On the other hand,
\[
\begin{aligned}
|T_2(\xi_2)-T_2(\xi_1)|&\leq\int_{\partial D_R(0)}\left|\frac{h(z)}{z^k}\right|\left|\frac{1}{z-\xi_2}-\frac{1}{z-\xi_1}\right||dz|\\
&=\frac{|\xi_2-\xi_1|}{R^{k-1}}\int_0^{2\pi}\frac{|h(Re^{i\theta})|}{|Re^{i\theta}-\xi_2||Re^{i\theta}-\xi_2|}d\theta\\
&\leq \frac{|\xi_2-\xi_1|}{R^{k-1}(R-|\xi_2|)(R-|\xi_2|)}\int_0^{2\pi}|h(Re^{i\theta})|d\theta. 
\end{aligned}
\]
Therefore, by taking Cauchy sequences again, there exists $\lim_{\xi\rightarrow0}T_2(\xi).$ Since, by the Cauchy-Pompeiu formula (\ref{l-2}), p. \pageref{l-2},
\[
2\pi ih(\xi)\xi^{-k} = T_1(\xi) + T_2(\xi),
\]
there exists $\lim_{\xi\rightarrow0}h(\xi)\xi^{-k}.$

\emph{Step 3. If $\lim_{z\rightarrow 0} h(z)/z^k=0$ for every $k\in\mathbb{N},$ then $h=0$ in some neighborhood of $z=0.$} 

Suppose, by contradiction, there exists $z_0$ in a neighborhood of $0,$ $|z_0|<R,$ such that $h(z_0)\neq 0.$ Taking the power $q$ and integrating (\ref{eq-q3}), p. \pageref{eq-q3}, over $\xi=x+iy,$ using Fubini's theorem and (\ref{qq-2}), p. \pageref{qq-2}, we have
\begin{equation}\label{eq-D}
\begin{aligned}
(2\pi)^q\int_{D_R(0)}\left|\frac{h(\xi)}{\xi^k}\right|^q dx\wedge dy &\leq C_1\int_{D_R(0)}\left[\int_{\partial D_R(0)}\left|\frac{h(z)}{z^k(z-\xi)}\right|^q|dz|\right]dx\wedge dy\\
&\leq C_1\int_{\partial D_R(0)}\left|\frac{h(z)}{z^k}\right|^q\left[\int_{D_R(0)}\frac{dx\wedge dy}{|\xi - z|^q}\right]|dz|\\
&\leq \frac{2\pi (2R)^{2-q}C_1}{2-q}\int_{\partial D_R(0)}\left|\frac{h(z)}{z^k}\right|^q|dz|.
\end{aligned}
\end{equation}
Let
\begin{equation}\label{D-star}
D^*=\left\{z\in D_R(0);|z|\leq|z_0|\ \mbox{and}\ |h(z)|\geq\frac{|h(z_0)|}{2}\right\},
\end{equation}
see Figure \ref{fig-3}.
\begin{figure}[h]
\centering
\includegraphics[scale=0.4]{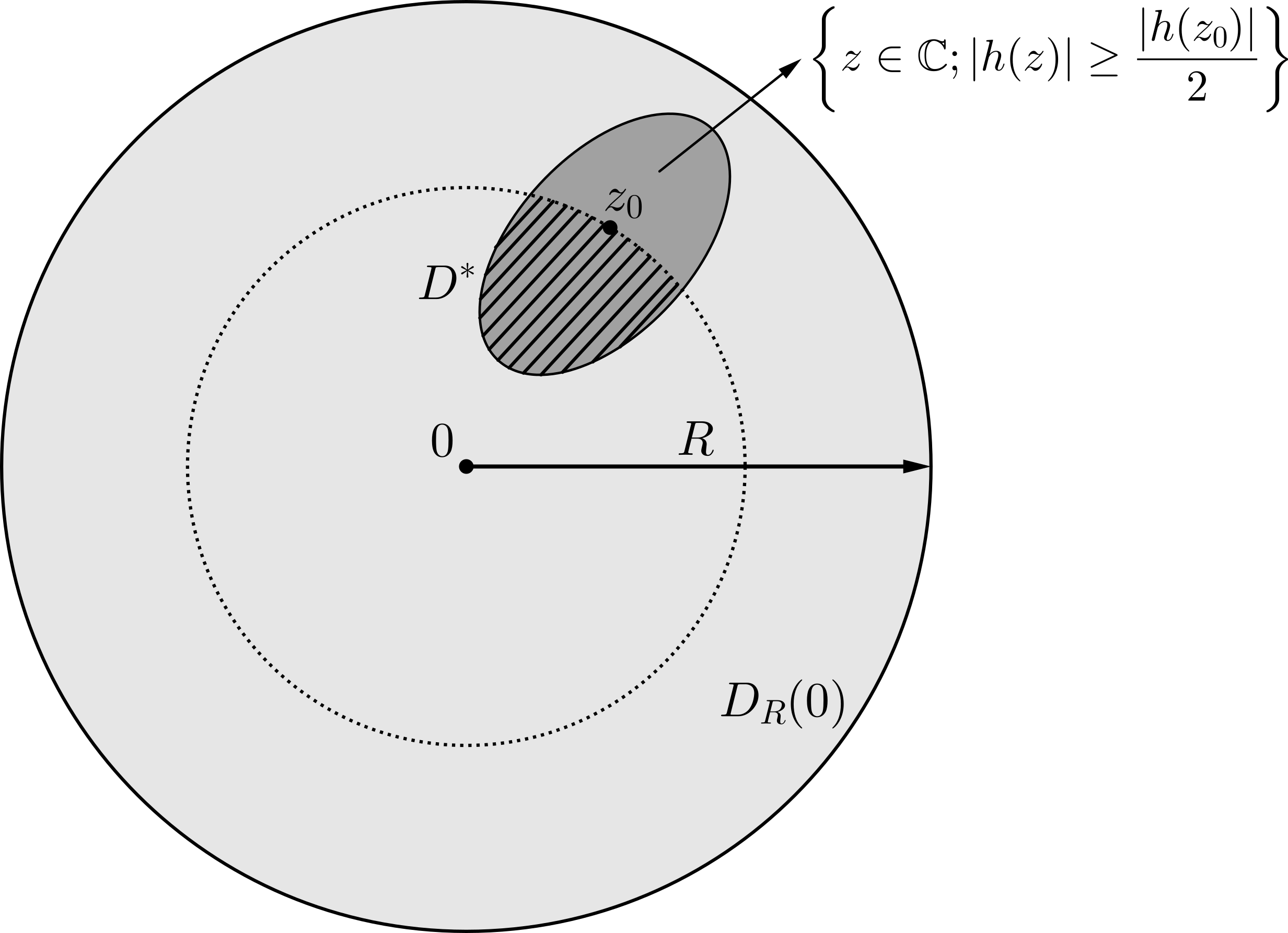}
\caption{Representation of the set $D^*,$ see (\ref{D-star}).}
\label{fig-3}
\end{figure}

On the one hand,
\begin{equation}\label{est-q1}
(2\pi)^q\int_{D_R(0)}\left|\frac{h(\xi)}{\xi^k}\right|^q dx\wedge dy \geq \left|\frac{h(z_0)}{z_0^k}\right|^q\pi^q \vol D^*=: a|z_0|^{-qk}.
\end{equation}
On the other hand, 
\begin{equation}\label{est-q2}
\frac{2\pi (2R)^{2-q}C_1}{2-q}\int_{\partial D_R(0)}\left|\frac{h(z)}{z^k}\right|^q|dz| = \left[\frac{\pi (2R)^{3-q}C_1}{2-q}\int_0^{2\pi}|h(Re^{i\theta})|^q d\theta\right]R^{-qk}=:bR^{-qk}.
\end{equation}
Replacing both (\ref{est-q1}) and (\ref{est-q2}) in (\ref{eq-D}) gives $a|z_0|^{-qk}\leq b R^{-qk}.$ Since $|z_0|<R,$ we have
\[
0\leq \limsup_{k\rightarrow\infty}\frac{a}{b} \leq \lim_{k\rightarrow\infty}\left(\frac{|z_0|^q}{R^q}\right)^k=0,
\]
i.e., $a=0.$ But since $a=|h(z_0)|^q\pi^q\vol D^*,$ we conclude that $|h(z_0)|=0,$ which is a contradiction. Therefore $h=0$ in a neighborhood of $z=0.$

\emph{Step 4. Conclusion.} If $h$ is not identically zero in a neighborhood of $z=0$ then, by Step 3, there exists $k>0$ such that $\lim_{k\rightarrow 0}h(z)z^{-(k-1)}=0$ and $\lim_{k\rightarrow 0}h(z)z^{-k}=c\neq 0$ or $\lim_{k\rightarrow 0}h(z)z^{-k}$ do not exists. But, by the Step 2, the second case cannot happen and thus there exists $c\in\CC$ such that
\[
\lim_{z\rightarrow 0}\frac{h(z)}{z^k}=c\neq 0.
\]
This implies that
\[
\frac{h(z)}{z^k}=c + R, \ \mbox{with} \ \lim_{z\rightarrow 0} R=0,
\]
i.e.,
\[
h(z) = z^k(c+R)=:z^k h_k(z) \ \mbox{with} \ h_k(0)=c\neq 0.
\]
Therefore, we conclude the proof of the theorem.
\end{proof}

\section{Proof of the ridigity theorems}\label{H-f}

Before proving our main theorems, we  give a brief introduction to weighted geometry in $\R^n.$ We refer, for example, \cite{C-M-Z} for a more detailed exposition. We call ($\R^n,\lan\cdot,\cdot\ran,e^{-f}$) a weighted Riemannian manifold if it has a weighted measure $dV_f = e^{-f}dV,$ where $f:\R^n\rightarrow\R$ is a function of class $C^2$. Let $X:\Sigma\rightarrow\R^n$ be an immersion of a surface $\Sigma.$ Consider $\Sigma$ with the weighted measure 
\[
 d\Sigma_f = e^{-f}d\Sigma,
\]
and the induced metric $\lan\cdot,\cdot\ran.$ 

The first variation of the weighted volume $V_f(\Sigma)=\int_\Sigma e^{-f}d\Sigma$ is given by
\[
\left.\dfrac{d}{dt}V_f(\Sigma_t)\right|_{t=0} = -\int_\Sigma \lan T^{\perp}, {\bf H}_f\ran e^{-f}d\Sigma,
\]
where $T$ is a compactly supported variational vector field on $\Sigma$ and
\begin{equation}\label{H_F}
{\bf H}_f = {\bf H} + (\n f)^\perp 
\end{equation}
is the weighted mean curvature vector of $\Sigma$ in $\R^n.$ Here $(\n f)^\perp$ denotes the part of the gradient $\n f $ of $f$ in $\R^n$ normal to $\Sigma$ and ${\bf H}$ denotes the non-normalized mean curvature vector of $\Sigma$ in $\R^n,$ i.e., the trace of the operator
\[
B(Z,W) = \n_ZW - \n^\Sigma_ZW,
\]
where $\n$ and $\n^\Sigma$ denote the connection of $\R^n$ and $\Sigma,$ respectively.

We say that a surface $\Sigma$ has parallel weighted mean curvature, if ${\bf H}_f$ is parallel in the normal bundle, i.e., $\n^\perp{\bf H}_f=0.$ In particular, if ${\bf H}_f=0,$ we say that $\Sigma$ is $f$-minimal.

In the case that $f(X)=\|X\|^2/4,$ we call the weighted manifold ($\R^n,\lan\cdot,\cdot\ran,e^{-\|X\|^2/4}$) the Gaussian space. Notice that self-shrinkers are $f$-minimal surfaces in the Gaussian space.

If the codimension is one, the parallel weighted mean curvature surfaces in the Gaussian space are called $\lambda$-surfaces. By using (\ref{H_F}), we can see that $\lambda$-surfaces are characterized by the equation
\[
\lambda = H + \frac{1}{2}\lan X,N\ran,
\]
where $\lambda\in\R,$ $N$ is the unit normal vector field of the immersion, and $H$ is its mean curvature, i.e., ${\bf H}=HN.$ Observe that self-shrinkers of $\R^3$ are also $\lambda$-surfaces for $\lambda=0.$

For each point $p\in\Sigma$ we can take isothermal parameters $u$ and $v$ in a neighborhood of $p,$ i.e., \[ds^2=\alpha(u,v)(du^2+dv^2),\] where $ds^2$ is the metric of $\Sigma$ and $\alpha(u,v)$ is a positive smooth function on $\Sigma$. Complexifying the parameters by taking $z=u+iv,$ we can identify $\Sigma$ with a subset of $\mathbb{C}.$ In this case, we have
\[
\lan X_z,X_{\bar{z}}\ran = \frac{\alpha(z)}{2} \ \mbox{and} \ ds^2=\alpha(z)|dz|^2.
\]
%and $ds^2=\alpha(z)|dz|^2.$

The immersion $X$ satisfies the equations
\begin{equation}\label{codazzi}
\left\{
\begin{aligned}
\n_{X_z}X_z&=\frac{\alpha_z}{\alpha}X_z + B(X_z,X_z),\\
\n_{X_{\bar{z}}}X_z&=\frac{\alpha}{4}{\bf H},\\
\n_{X_{\bar{z}}}X_{\bar{z}}&= \frac{\alpha_{\bar{z}}}{\alpha}X_{\bar{z}} + B(X_{\bar{z}},X_{\bar{z}}),\\
\end{aligned}
\right.
\end{equation}
and, for any $\nu\in T\Sigma^\perp,$
\begin{equation}\label{codazzi-2}
\left\{
\begin{aligned}
\n_{X_z}\nu &=-\frac{1}{2}\lan{\bf H},\nu\ran X_z - \frac{2}{\alpha}\lan B(X_z,X_z),\nu\ran X_{\bar{z}} + \n^\perp_{X_z}\nu\\
\n_{X_{\bar{z}}}\nu&= - \frac{2}{\alpha}\lan B(X_{\bar{z}},X_{\bar{z}}),\nu\ran X_z - \frac{1}{2}\lan{\bf H},\nu\ran X_{\bar{z}} + \n^\perp_{X_{\bar{z}}}\nu,\\
\end{aligned}
\right.
\end{equation}
where $\n^\perp$ is the connection of the normal bundle $T\Sigma^\perp.$

Let us denote by
\[
P^\nu dz^2 = \lan B(X_z,X_z),\nu\ran dz^2
\]
the $(2,0)$-part of the second fundamental form of $\Sigma$ in $\R^n$ relative to the normal $\nu\in T\Sigma^\perp.$ This quadratic form is also called the Hopf quadratic differential.

Since 
\begin{equation}\label{hopf-diff-000}
\begin{aligned}
P^\nu&=\lan \n_{X_z}X_z,\nu \ran = \frac{1}{4}\lan \n_{X_u-iX_v}X_u-iX_v,\nu\ran\\
 &=\frac{1}{4}[\lan \n_{X_u}X_u,\nu \ran - \lan \n_{X_v}X_v,\nu\ran - i(\lan \n_{X_u}X_v,\nu\ran + \lan \n_{X_v}X_u,\nu\ran)]\\
 &=\frac{1}{4}[II^\nu(X_u,X_u) - II^\nu(X_v,X_v) - 2i II^\nu(X_u,X_v)],\\
\end{aligned}
\end{equation}
where $II^\nu$ is the second fundamental form of $\Sigma$ in $\R^n$ relative to $\nu\in T\Sigma^\perp,$ we have $P^\nu=0$ if and only if $II^\nu$ is umbilical.
 
The next result will be an important tool to the proof of the main results.

\begin{proposition}\label{main-prop}
Let $\Sigma$ be a Riemann surface and $P^\nu dz^2 = \lan \n_{X_z}X_z,\nu\ran dz^2$ be the Hopf differential, relative to $\nu\in T\Sigma^\perp,$ of an immersion $X:(\Sigma,\alpha(z)|dz|^2)\to \R^n.$ Define 
\[
Q^\nu dz^2 = e^{-\frac{1}{2}f}P^\nu dz^2.
\]
If $\nu$ is parallel at the normal bundle, i.e., $\n^\perp\nu= 0,$ then
\[
Q^\nu_{\bar{z}}=\frac{\alpha}{4}e^{-\frac{1}{2}f}\left[\lan {\bf H}_f,\nu\ran_z - \hess f(X_z,\nu)+ \frac{1}{2}\lan {\bf H}_f - \n f,\nu\ran\lan \n f, X_z\ran\right],
\]
where $\hess f$ is the hessian of $f.$% and $\n f = \frac{\partial f}{\partial x_1}e_1 + \frac{\partial f}{\partial x_2}e_2+\frac{\partial f}{\partial x_3}e_3$ is the gradient of $f$ in $\R^3,$ and the dot means the usual scalar product of $\R^3.$
\end{proposition}
\begin{proof}
First let us prove that, for $\n^\perp\nu=0,$ we have
\[
P^\nu_{\bar{z}} = \frac{\alpha}{4}\lan {\bf H},\nu\ran_z.
\]
In fact, by using (\ref{codazzi}) and (\ref{codazzi-2}),
\[
\begin{aligned}
P^\nu_{\bar{z}}&= \frac{\partial}{\partial \bar{z}}\lan \n_{X_z}X_z,\nu\ran = \lan\n_{X_{\bar{z}}}\n_{X_z}X_z,\nu\ran + \lan\n_{X_z}X_z,\n_{X_{\bar{z}}}\nu\ran\\
&=\lan R(X_z,X_{\bar{z}})X_z,\nu\ran + \lan\n_{X_z}\n_{X_{\bar{z}}}X_z,\nu\ran + \lan\n_{X_z}X_z,\n_{X_{\bar{z}}}\nu\ran\\
&=\frac{\partial}{\partial z}\left(\lan\n_{X_{\bar{z}}}X_z ,\nu\ran\right) - \lan\n_{X_{\bar{z}}}X_z,\n_{X_z}\nu\ran + \lan\n_{X_z}X_z,\n_{X_{\bar{z}}}\nu\ran\\
&= \frac{\partial}{\partial z}\left(\frac{\alpha }{4}\lan{\bf H},\nu\ran\right)- \left\lan\frac{\alpha}{4}{\bf H}, - \frac{1}{2}\lan{\bf H},\nu\ran X_z - \frac{2P^\nu}{\alpha}X_{\bar{z}} + \n^\perp_{X_z}\nu\right\ran\\
&\qquad + \left\lan\frac{\alpha_z}{\alpha}X_z + B(X_z,X_z),-\frac{2\overline{P^\nu}}{\alpha}X_z - \frac{1}{2}\lan{\bf H},\nu\ran X_{\bar{z}} + \n_{X_{\bar{z}}}^\perp\nu\right\ran\\
&=\frac{\alpha}{4}[\lan{\bf H},\nu\ran_z - \lan{\bf H},\n^\perp_{X_z}\nu\ran] + \lan B(X_z,X_z),\n_{X_{\bar{z}}}^\perp\nu\ran\\
&=\frac{\alpha}{4}\lan{\bf H},\nu\ran_z,
\end{aligned}
\]
where $R(X_z,X_{\bar{z}})X_z=0$ is the Euclidean curvature tensor and, in the last equality, we used that $\n^\perp\nu=0$. Since
\[
\begin{aligned}
Q^\nu_{\bar{z}}&= \frac{\partial}{\partial\bar{z}}(e^{-\frac{1}{2}f}P^\nu) = -\frac{1}{2}f_{\bar{z}}e^{-\frac{1}{2}f}P^\nu + e^{-\frac{1}{2}f}P^\nu_{\bar{z}}\\
&=-\frac{P^\nu}{2}e^{-\frac{1}{2}f}\lan\n f,X_{\bar{z}}\ran + \frac{\alpha}{4}e^{-\frac{1}{2}f}\lan{\bf H},\nu\ran_z
\end{aligned}
\]
and ${\bf H} = {\bf H}_f - (\n f)^\perp,$ we have

\[
\begin{aligned}
Q^\nu_{\bar{z}}&=-\frac{P^\nu}{2}e^{-\frac{1}{2}f}\lan\n f,X_{\bar{z}}\ran + \frac{\alpha}{4}e^{-\frac{1}{2}f}\lan {\bf H}_f - (\n f)^\perp, \nu\ran_z\\
&=-\frac{P^\nu}{2}e^{-\frac{1}{2}f}\lan\n f,X_{\bar{z}}\ran + \frac{\alpha}{4}e^{-\frac{1}{2}f}[\lan {\bf H}_f, \nu\ran_z - \lan \n f, \nu\ran_z]\\
&=-\frac{P^\nu}{2}e^{-\frac{1}{2}f}\lan\n f,X_{\bar{z}}\ran + \frac{\alpha}{4}e^{-\frac{1}{2}f}[\lan {\bf H}_f, \nu\ran_z - \lan \n_{X_z}\n f, \nu\ran - \lan \n f, \n_{X_z}\nu\ran]\\
&=-\frac{P^\nu}{2}e^{-\frac{1}{2}f}\lan\n f,X_{\bar{z}}\ran + \frac{\alpha}{4}e^{-\frac{1}{2}f}\left[\lan {\bf H}_f, \nu\ran_z - \hess f(X_z,\nu)- \left\lan \n f, - \frac{1}{2}\lan{\bf H},\nu\ran X_z - \frac{2P^\nu}{\alpha}X_{\bar{z}}\right\ran\right]\\
&=\frac{\alpha}{4}e^{-\frac{1}{2}f}\left[\lan {\bf H}_f, \nu\ran_z - \hess f(X_z,\nu)+ \frac{1}{2}\lan{\bf H},\nu\ran\lan \n f, X_z\ran\right]\\
&=\frac{\alpha}{4}e^{-\frac{1}{2}f}\left[\lan {\bf H}_f, \nu\ran_z - \hess f(X_z,\nu)+ \frac{1}{2}\lan {\bf H}_f - \n f,\nu\ran\lan \n f, X_z\ran\right],
\end{aligned}
\]
where, in the fourth equality, we used again (\ref{codazzi-2}) and $\n^\perp\nu=0.$
\end{proof}

We will also need the following result which proof can be found essentially in Yau \cite{Yau} (see Theorem 1, p. 351-352) and Chen-Yano \cite{Chen-Yano} (see Theorem 3.3, p. 472-473). For the conclusion when $\nu={\bf H}/\|{\bf H}\|$ we use Theorem 2, p. 117, of the work of Ferus \cite{Ferus}.

\begin{lemma}\label{CYY}
Let $X:\Sigma\to\R^{2+m},$ $m\geq 1,$ be an immersion of surface homeomorphic to the sphere. If there exists a normal vector field $\nu\in T\Sigma^\perp$ such the that $\n^\perp\nu\equiv 0$ and $A^\nu=\mu I$ everywhere in $\Sigma,$ where $A^\nu$ is the shape operator of the second fundamental form of $X$ relative to $\nu,$ then $\mu$ is constant and 
\begin{itemize}
\item[i)] $X(\Sigma)$ is contained in a round hypersphere of $\R^{2+m}$, if $\mu\neq 0;$
\item[ii)] or $X(\Sigma)$ is contained in a hyperplane of $\R^{2+m},$ if $\mu=0.$
\end{itemize}
Moreover, if the mean curvature vector ${\bf H}\neq 0$ and $\nu={\bf H}/\|{\bf H}\|,$ then $X$ has parallel mean curvature and $X(\Sigma)$ is a minimal surface of a hypersphere of $\R^{2+m}.$
\end{lemma}
\begin{proof}
Since $\n^\perp \nu\equiv 0,$ we have
\[
\n_U\nu= \n^\perp_U\nu + A^\nu(U) = \mu U,
\]
for every $U\in T\Sigma.$ Taking the covariant derivative, we obtain
\begin{equation}\label{umb-1}
\n_V\n_U\nu = \n_V(\mu U) = V(\mu)U + \mu\n_VU.
\end{equation}
for every $V\in T\Sigma.$ On the other hand, in $\R^{2+m},$
\begin{equation}\label{umb-2}
\begin{aligned}
\n_V\n_U\nu&= \n_U\n_V\nu + \n_{[V,U]}\nu\\
&= \n_U(\mu V) +\mu[V,U]\\
&= U(\mu)V+\mu \n_UV + \mu \n_VU - \mu \n_UV\\
&= U(\mu)V+\mu \n_VU.
\end{aligned}
\end{equation}
Comparing \eqref{umb-1} and \eqref{umb-1}, we obtain
\[
V(\mu)U=U(\mu)V.
\]
Since $U$ and $V$ can be taken linearly independent, we conclude that $U(\mu)=0$ for every $U\in T\Sigma,$ i.e., $\mu$ is constant. If $\mu=0,$ then $\n_U\nu\equiv 0.$ This implies
\[
U\lan X,\nu\ran = \lan \n_UX,\nu\ran + \lan X,\n_U\nu\ran = \lan U,\nu\ran =0,
\]
i.e., $\lan X,\nu\ran$ is constant and $X(\Sigma)$ lies in a hyperplane with normal $\nu.$ On the other hand, if $\mu\neq 0,$ then $Y=X-\frac{1}{\mu}\nu$ satisfies
\[
\n_UY = \n_UX - \frac{1}{\mu}\n_U\nu = U - \frac{1}{\mu}(\mu U)=0,
\]
which implies that $Y$ is a constant vector $x_0,$ i.e., $\|X-x_0\|^2=1/\mu^2$ and $X(\Sigma)$ lies in a hypersphere $\mathbb{S}^{1+m}(x_0,1/\mu)$ of $\R^{2+m}$ with center $x_0$ and radius $1/\mu.$ 

Now, assume ${\bf H}\neq 0$ and $\nu={\bf H}/\|{\bf H}\|.$ Let $\{\nu,\eta_2,\ldots,\eta_m\}$ be an orthornormal frame of $T\Sigma^\perp.$ We have
\[
{\bf H} = (\tr A^\nu)\nu + \sum_{i=2}^m (\tr A^{\eta_i})\eta_i.
\]
This implies that $\tr A^{\eta i}\equiv 0,$ $i=2,\ldots,m.$ Since $\tr A^\nu = 2\mu,$ which is constant, we have that ${\bf H}=2\mu\nu$ is parallel, i.e., $X$ has parallel mean curvature vector. The conclusion then comes from Theorem 2, p.117 of \cite{Ferus}, which states that if a surface, homeomorphic to the sphere, is immersed in some Euclidean space, with parallel nonzero mean curvature vector, then $X$ immerses $\Sigma$ as a minimal submanifold of some Euclidean hypersphere.
\end{proof}

Now, we are ready to state and prove the main theorem of this section. This theorem is a rigidity result for parallel weighted mean curvature ${\bf H}_f$ surfaces in the Euclidean space with arbitrary codimension and radial weight $f(X)=F(\|X\|^2),$ where $F:\R\to\R$ is a function of class $C^2$.  {\it Since the codimension can be  arbitrary large, we  assume that $X(\Sigma)$ does not lie in any proper affine subspace of the Euclidean space.}

\begin{theorem}\label{Gauss-space}
Let $X:\Sigma\to\R^{2+m},\ m\geq1,$ be an immersion of a surface homeomorphic to the sphere. Assume that all the following assertions holds:
\begin{itemize}
\item[i)] $X$ has parallel weighted mean curvature ${\bf H}_f,$ i.e., $\n^\perp{\bf H}_f=0,$ for a radial weight $f(X)=F(\|X\|^2)$, where $F:\R\to\R$ is a function of class $C^2$.
\item[ii)] There exists a unitary normal vector field $\nu\in T\Sigma^\perp$ such that $\n^\perp\nu=0.$
\item[iii)] There exists a non-negative locally $L^p$ function $\vp:\Sigma\to\R,\ p>2,$ and a locally integrable function $G:[0,\infty)\rightarrow[0,\infty)$ satisfying $\limsup_{t\rightarrow 0}G(t)/t<\infty,$ such that  
\begin{equation}\label{HG}
\left|F'(\|X\|^2)\lan {\bf H}_f,\nu\ran - 2\left[2F''(\|X\|^2)+(F'(\|X\|^2))^2\right]\lan X,\nu\ran\right|\|X^\top\|\leq \vp G(\|\Phi^\nu\|).
\end{equation}
\end{itemize}
Then $X(\Sigma)$ is contained in a round hypersphere of $\R^{2+m}.$ Moreover, if ${\bf H}\neq 0$ and $\nu={\bf H}/\|{\bf H}\|,$ then $X(\Sigma)$ is a minimal surface of a round hypersphere of $\R^{2+m}$ or a round sphere in $\R^{2+m}.$ 

Here $X^\top$ denotes the component of $X$ tangent to $T\Sigma,$ $\|\Phi^\nu\|$ denotes the matrix norm of $\Phi^\nu=A^\nu-(\tr A^\nu/2)I,$ where $A^\nu$ is the shape operator of the second fundamental form of $X$ relative to $\nu,$ $\tr A^\nu$ is its trace, and $I:T\Sigma\to T\Sigma$ is the identity operator.
\end{theorem}

\begin{proof}
First, notice that, since $e_1=(1/\sqrt{\alpha})X_u$ and $e_2=(1/\sqrt{\alpha})X_v$ forms an orthonormal frame for $T\Sigma,$ denoting by $h_{ij}^\nu=II^\nu(e_i,e_j),$ by using (\ref{hopf-diff-000}), p. \pageref{hopf-diff-000}, we have
%\begin{equation}\label{norm-II}
\[
\begin{aligned}
\|\Phi^\nu\|^2&=(h_{11}^\nu-(\tr II^\nu/2))^2 + (h_{22}^\nu-(\tr II^\nu /2))^2 + 2(h_{12}^\nu)^2\\
              &=2\left(\frac{h_{11}^\nu-h_{22}^\nu}{2}\right)^2 + 2(h_{12}^\nu)^2\\
              &=\frac{1}{2}\left[(h_{11}^\nu-h_{22}^\nu)^2 + 4(h_{12}^\nu)^2\right]\\
              &=\frac{1}{2\alpha^2}\left[(II^\nu(X_u,X_u) - II^\nu(X_v,X_v))^2 + 4II^\nu(X_u,X_v)^2\right]\\
              &=\frac{8}{\alpha^2}|P^\nu|^2.
\end{aligned}
\]
This gives
\[
|Q^\nu|= e^{-\frac{1}{2}F(\|X\|^2)}|P^\nu|=\frac{\alpha}{2\sqrt{2}}e^{-\frac{1}{2}F(\|X\|^2)}\|\Phi^\nu\|.
\]
On the other hand,
\[
\nabla f= 2F'(\|X\|^2)X \ \mbox{and} \ \frac{\partial^2 f}{\partial x_i\partial x_j}=4F''(\|X\|^2)x_ix_j + 2F'(\|X\|^2)\delta_{ij},\\
\]
where $\delta_{ij}=1,$ if $i=j,$ and $\delta_{ij}=0,$ if $i\neq j.$ By using Proposition \ref{main-prop}, p. \pageref{main-prop}, we have
\[
\begin{aligned}
Q^\nu_{\bar{z}}&=\frac{\alpha}{4}e^{-\frac{1}{2}F(\|X\|^2)}[-4F''(\|X\|^2)\lan X,X_z\ran\lan X,\nu\ran + (\lan{\bf H}_f,\nu\ran - 2F'(\|X\|^2)\lan X,\nu\ran)F'(\|X\|^2)\lan X,X_z\ran]\\
&=\frac{\alpha}{4}e^{-\frac{1}{2}F(\|X\|^2)}[F'(\|X\|^2)\lan{\bf H}_f,\nu\ran - 2(2F''(\|X\|^2)+(F'(\|X\|^2))^2)\lan X,\nu\ran]\lan X,X_z\ran,
\end{aligned}
\]
provided $\n^\perp{\bf H}_f=0$ and $\n^\perp\nu=0$ imply that $\lan{\bf H}_f,\nu\ran$ is constant. Since
\[
X=\frac{2}{\alpha}\lan X,X_{\bar{z}}\ran X_z + \frac{2}{\alpha}\lan X,X_z\ran X_{\bar{z}} + X^\perp,
\]
where $X^\perp$ is the part of $X$ normal to $\Sigma,$ and $|\lan X,X_z\ran|=\frac{1}{2}|\lan X,X_u\ran - i \lan X,X_v\ran|=|\lan X,X_{\bar{z}}\ran|,$ we have
\[
\|X^\top\|=\frac{2}{\sqrt{\alpha}}\sqrt{|\lan X,X_z\ran||\lan X,X_{\bar{z}}\ran|}=\frac{2}{\sqrt{\alpha}}|\lan X,X_z\ran|.
\]
Thus, by using hypothesis (\ref{HG}), we obtain 
\[
\begin{aligned}
|Q_{\bar{z}}^\nu|&\leq \frac{\alpha}{4}e^{-\frac{1}{2}F(\|X\|^2)}\left|F'(\|X\|^2)\lan{\bf H}_f,\nu\ran - 2\left[2F''(\|X\|^2) + (F'(\|X\|^2))^2\right]\lan X, \nu \ran\right||\lan X,X_z\ran|\\
&\leq \frac{\alpha^{3/2}}{8}e^{-\frac{1}{2}F(\|X\|^2)}\left|F'(\|X\|^2)\lan{\bf H}_f,\nu\ran - 2\left[2F''(\|X\|^2) + (F'(\|X\|^2))^2\right]\lan X, \nu\ran\right|\|X^\top\|\\
&\leq\frac{\alpha^{3/2}}{8}e^{-\frac{1}{2}F(\|X\|^2)} \vp G(\|\Phi^\nu\|)\\
&\leq\frac{\alpha^{3/2}}{8}e^{-\frac{1}{2}F(\|X\|^2)} \vp G\left(\frac{2\sqrt{2}}{\alpha}e^{\frac{1}{2}F(\|X\|^2)}|Q^\nu|\right).\\
\end{aligned}
\]
Define 
\[
h(z)=2\sqrt{2}\alpha^{-1}e^{\frac{1}{2}F(\|X\|^2)}Q^\nu=2\sqrt{2}\alpha^{-1}P^\nu.
\] 
We have
\[
\begin{aligned}
\left|\frac{\partial h}{\partial \bar{z}}\right|&\leq \left|\frac{\partial}{\partial \bar{z}}\left(2\sqrt{2}\alpha^{-1}e^{\frac{1}{2}F(\|X\|^2)}\right)\right||Q^\nu| + 2\sqrt{2}\alpha^{-1}e^{\frac{1}{2}F(\|X\|^2)}|Q^\nu_{\bar{z}}|\\
&\leq\left|\frac{\partial}{\partial \bar{z}}\left(2\sqrt{2}\alpha^{-1}e^{\frac{1}{2}F(\|X\|^2)}\right)\right||Q^\nu| + \sqrt{\frac{\alpha}{8}}\vp G(|h(z)|)\\
&=\frac{\left|-\alpha_{\bar{z}}\alpha^{-2}e^{\frac{1}{2}F(\|X\|^2)}+\alpha^{-1}\frac{1}{2}F'(\|X\|^2)(\|X\|^2)_{\bar{z}}e^{\frac{1}{2}F(\|X\|^2)}\right|}{\alpha^{-1}e^{\frac{1}{2}F(\|X\|^2)}}|h(z)|+\sqrt{\frac{\alpha}{8}}\vp G(|h(z)|)\\
&\leq \left[|\alpha_{\bar{z}}|\alpha^{-1}+\frac{\sqrt{\alpha}}{2}F'(\|X\|^2)\|X^\top\|+\sqrt{\frac{\alpha}{8}}\vp\right]\widetilde{G}(|h(z)|),
\end{aligned}
\]
where $\widetilde{G}(t)=\max\{t,G(t)\}.$ Since
\[
|\alpha_{\bar{z}}|\alpha^{-1}+\frac{\sqrt{\alpha}}{2}F'(\|X\|^2)\|X^\top\|+\sqrt{\frac{\alpha}{8}}\vp \in L^p_{loc},\ p>2,
\]
and
\[
\limsup_{t\to 0}\frac{\widetilde{G}(t)}{t}=\max\left\{1,\limsup_{t\to 0}\frac{G(t)}{t}\right\}<\infty,
\]
we are under the conditions of Theorem \ref{main}, p. \pageref{main}. Thus either $h(z),$ and thus $P^\nu,$ is identically zero in a neighborhood $V$ of a zero $z_0$ or this zero is isolated and the index of a direction field determined by $\textrm{Im}[P^\nu dz^2]=0$ is $-k/2,$ hence negative. If, for some coordinate neighborhood $V$ of zero, $P^\nu=0$, this holds for the whole $\Sigma$. Otherwise, the zeroes on the boundary of $V$ will contradict to Theorem \ref{main}. So if $P^\nu$ is not identically zero, all zeroes, if any, are isolated and have negative indices. This implies that the sum of all indexes of the isolated zeroes are negative (if there are zeroes) or zero (if there are no zeroes). Since $\Sigma$ has genus zero, by the Poincar\'e index theorem the sum of the indices of the singularities of any field of directions is $2$ (hence positive). This contradiction shows that $P^\nu$ is identically zero. This implies that $A^\nu = \mu I,$ i.e., $\nu$ is a umbilical normal direction of $X$. By using Lemma \ref{CYY}, since $X(\Sigma)$ does not lie in a hyperplane, we conclude that $\mu\neq 0$ and $X(\Sigma)$ lies in a hypersphere of $\R^{2+m}.$ Moreover, if ${\bf H}\ne 0$ and $\nu={\bf H}/\|{\bf H}\|,$ by the same Lemma, $X(\Sigma)$ is a minimal surfaces of a hypersphere of $\R^{2+m}.$
\end{proof}

In the case when $\Sigma$ is $f$-minimal, i.e., ${\bf H}_f=0,$ and the weight $f(X)=F(\|X\|^2)$ satisfies $F'(t)\neq 0$ and $2F''(t) + (F'(t))^2\neq 0,$ for every $t\in\R,\ t\geq0,$ the next result follows from Theorem \ref{Gauss-space}.

\begin{corollary}\label{Gauss-space-cor}
Let $X:\Sigma\to\R^{2+m},\ m\geq1,$ be an immersion of a surface homeomorphic to the sphere. Assume that all the following assertions holds:
\begin{itemize}
\item[i)] $X$ is $f$-minimal, i.e., ${\bf H}_f=0,$ for a radial weight $f(X)=F(\|X\|^2)$, where $F:\R\to\R$ is a function of class $C^2$ such that $F'(t)\neq0$ and $2F''(t) + (F'(t))^2\neq 0,$ for every $t\in\R, \ t\geq 0.$
\item[ii)] There exists an unitary normal vector field $\nu\in T\Sigma^\perp$ such that $\n^\perp\nu=0.$
\item[iii)] There exists a non-negative locally $L^p$ function $\vp:\Sigma\to\R,\ p>2,$  and  a locally integrable function $G:[0,\infty)\rightarrow[0,\infty)$ satisfying $\limsup_{t\rightarrow 0}G(t)/t<\infty,$ such that  
\begin{equation}\label{HG-cor}
\left(\|X\|^2 - \left(\frac{\|{\bf H}\|}{2F'(\|X\|^2)}\right)^2\right)\left(\frac{|\lan{\bf H},\nu\ran|}{2F'(\|X\|^2)}\right)^2\leq \vp^2 G(\|\Phi^\nu\|)^2.
\end{equation}
\end{itemize}

Then $X(\Sigma)$ is contained in a round hypersphere of $\R^{2+m}$ of radius $R,$ where $R$ is the solution of the equation 
\[
F'(R^2)R^2=1,
\]
and centered at the origin. Moreover, if ${\bf H}\neq 0$ and $\nu={\bf H}/\|{\bf H}\|,$ then $X(\Sigma)$ is a minimal surface of a round hypersphere of $\R^{2+m}$ with the same properties.

Here $\|\Phi^\nu\|$ is the matrix norm of $\Phi^\nu=A^\nu-(\tr A^\nu/2)I,$  where $A^\nu$ is the shape operator of the second fundamental form of $X$ relative to $\nu,$ $\tr A^\nu$ is its trace, and $I:T\Sigma\to T\Sigma$ is the identity operator.
\end{corollary}

\begin{proof}
By taking ${\bf H}_f=0$ in (\ref{HG}), we obtain
\begin{equation}\label{HG-cor-dem}
\left|\left\lan X,\nu\right\ran\right|\|X^\top\| \leq \dfrac{\varphi}{2F''(\|X\|^2) + (F'(\|X\|^2))^2}G(\|\Phi^\nu\|).
\end{equation}
Since, using (\ref{H_F}),
\[
0={\bf H}_f={\bf H}+ 2F'(\|X\|^2)X^\perp,
\]
we have
\[
\left\lan X,\nu\right\ran= - \frac{\lan {\bf H},\nu\ran}{2F'(\|X\|^2)}\ \mbox{and}\ \|X^\top\|^2=\|X\|^2 - \|X^\perp\|^2 = \|X\|^2 - \left(\frac{\|{\bf H}\|}{2F'(\|X\|^2)}\right)^2.
\]
Replacing these expressions in (\ref{HG-cor-dem}), considering $\varphi/(2F''(\|X\|^2) + (F'(\|X\|^2))^2)$ in the place of $\vp,$ and squaring the resultant inequality, we obtain that (\ref{HG}) becomes (\ref{HG-cor}). The result then follows from Theorem \ref{Gauss-space}.

In order to determine the radius and the center of the sphere, consider ${\bf H}_{\ss}$ the mean curvature vector of $\Sigma$ in $\ss^{1+m}(x_0,R),$ where $x_0$ is the center and $R$ is the radius of the sphere, and $\widetilde{II}$ the second fundamental form of $\ss^{1+m}(x_0,R)$ in $\R^{2+m}.$ We have
\[
{\bf H} = {\bf H}_{\ss}+\sum_{i=1}^2\widetilde{II}(e_i,e_i) \Rightarrow \lan{\bf H},\nu\ran= \lan{\bf H}_{\ss},\nu\ran + \frac{2}{R},
\]
where $\{e_1,e_2\}$ is an orthonormal frame of $T\Sigma.$ Since ${\bf H}_{\ss}\in T\ss^{1+m}(x_0,R),$ then $\lan{\bf H}_{\ss},\nu\ran=0,$ i.e.,
\[
\lan{\bf H},\nu\ran=\frac{2}{R}.
\] 
By using ${\bf H}=-2F'(\|X\|^2)X^\perp,$ we obtain
\[
RF'(\|X\|^2)\lan X,\nu\ran=-1.
\]
Since $X(\Sigma)\subset\ss^{1+m}(x_0,R),$ we have $X=x_0-R\nu.$ This gives
\begin{equation}\label{Hf0}
RF'(\|X\|^2)\lan X,x_0-X\ran=-1.
\end{equation}
Taking the gradient $\n_\Sigma$ of $\Sigma$ in (\ref{Hf0}),
\[
RF''(\|X\|^2)\lan X,x_0-X\ran\n_{\Sigma}(\|X\|^2) + F'(\|X\|^2)(\n_{\Sigma}\lan X,x_0\ran - \n_{\Sigma}(\|X\|^2))=0
\]
i.e.,
\begin{equation}\label{Hf1}
2RF''(\|X\|^2)\lan X,x_0-X\ran X^\top + F'(\|X\|^2)(x_0^\top - 2 X^\top)=0.
\end{equation}
Since $X^\top = x_0^\top,$ by multiplying (\ref{Hf1}) by $F'(\|X\|^2)$ we obtain
\[
\left[2F''(\|X\|^2)(RF'(\|X\|^2)\lan X,x_0-X\ran) - (F'(\|X\|^2))^2\right]X^\top =0.
\]
By using (\ref{Hf0}) again gives
\[
-\left[2F''(\|X\|^2)+(F'(\|X\|^2))^2\right]X^\top =0.
\]
The hypothesis $2F''(\|X\|^2)+(F'(\|X\|^2))^2\neq 0$ thus implies that $X^\top=0.$ Since $\n_\Sigma(\|X\|^2)=2X^\top=0,$ we have that $\|X\|^2$ is constant, i.e., $X$ is immersed in a sphere centered at the origin. On the other hand, calculating the Laplacian $\Delta_\Sigma$ of $\|X\|^2$ in $\Sigma$ gives
\[
0=\frac{1}{2}\Delta_\Sigma\|X\|^2 = \lan{\bf H},X\ran + 2 = -2F'(\|X\|^2)\|X\|^2+2, 
\]
i.e.,
\[
F'(\|X\|^2)\|X\|^2=1.
\]
\end{proof}

Since self-shrinkers are $f$-minimal surfaces for the weight $f(X)=\frac{\|X\|^2}{4},$ applying Corollary \ref{Gauss-space-cor} to $F(t)=t/4,$ we obtain

\begin{corollary}\label{codim}
Let $X:\Sigma\rightarrow\R^{2+m},\ m\geq1,$ be an immersed self-shrinker homeomorphic to the sphere. Assume there exists an unitary normal vector field $\nu\in T\Sigma^\perp$ such that $\n^\perp\nu=0.$ If there exists a non-negative locally $L^p$ function $\vp:\Sigma\to\R,\ p>2,$ and a locally integrable function $G:[0,\infty)\rightarrow[0,\infty)$ satisfying $\limsup_{t\rightarrow 0}G(t)/t<\infty,$ such that  
\[
\left(\|X\|^2 - 4\|{\bf H}\|^2\right)|\lan{\bf H},\nu\ran|^2\leq \vp^2 G(\|\Phi^\nu\|)^2,
\]
then $X(\Sigma)$ is contained in a round hypersphere of $\R^{2+m}$ of radius $2$ and centered at the origin.

Here $\|\Phi^\nu\|$ is the matrix norm of $\Phi^\nu=A^\nu-(\tr A^\nu/2)I,$ where $A^\nu$ is the shape operator of the second fundamental form of $X$ relative to $\nu,$ $\tr A^\nu$ is its trace, and $I:T\Sigma\to T\Sigma$ is the identity operator.
\end{corollary}

\begin{remark}
%{\normalfont 
In the particular case when ${\bf H}\neq 0$ and $\nu={\bf H}/\|{\bf H}\|,$ then Corollary \ref{codim} is also a consequence of the main theorem of Smoczyk, see \cite{smoc}: \emph{A closed $n$-dimensional self-shrinker of $\R^{n+m}$ is a minimal submanifold of the sphere $\ss^{n+m-1}(\sqrt{2n})$ if and only if ${\bf H}\neq 0$ and $\n^\perp\nu=0,$ where $\nu={\bf H}/\|{\bf H}\|.$}
%}
\end{remark}

If we consider the case of codimension one in Corollary \ref{codim}, then we obtain Theorem \ref{Gauss-space-2}:

\begin{corollary}[Theorem \ref{Gauss-space-2}, p. \pageref{Gauss-space-2}]
Let $X:\Sigma\to\R^3$ be an immersed self-shrinker homeomorphic to the sphere. If there exists a non-negative locally $L^p$ function $\vp:\Sigma\to\R,\ p>2,$ and a locally integrable function $G:[0,\infty)\rightarrow[0,\infty)$ satisfying $\limsup_{t\rightarrow 0}G(t)/t<\infty,$ such that  
\[
(\|X\|^2-4H^2)H^2\leq \vp^2 G(\|\Phi\|)^2,
\]
then $X(\Sigma)$ is a round sphere of radius $2$ and centered at the origin.

Here  $\|\Phi\|$ denotes the matrix norm of $\Phi=A-(H/2)I,$ where $A$ is the shape operator of the second fundamental form of $X,$ $H$ is its non-normalized mean curvature, and $I$ is the identity operator of $T\Sigma.$
\end{corollary}

Now, we present a proof of Corollary \ref{cor-1}, p. \pageref{cor-1}, as a consequence of Theorem \ref{Gauss-space-2}, p. \pageref{Gauss-space-2}.

\begin{corollary}[Corollary \ref{cor-1}, p. \pageref{cor-1}]\label{cor-1-1}
Let $X:\Sigma\to\R^3$ be an immersed self-shrinker homeomorphic to the sphere. If at each umbilical points, the lower order of $\|\Phi\|^2$ minus the upper order of the function $(\|X\|^2-4H^2)H^2$ is less than $2$, then $X(\Sigma)$ is a round sphere centered at the origin and of radius $2$.
\end{corollary}
\begin{proof}%[Proof of Corollary \ref{cor-1}]
Let $z_0\in\Sigma$ be a umbilical point, $a=\zeta_+^{(\|X\|^2-4H^2)H^2}(z_0)$ and $b=\zeta_-^{\|\Phi\|^2}(z_0)$. By using Definition \ref{defn-1}, p. \pageref{defn-1}, and the hypothesis, we have
\[
\limsup_{z\rightarrow z_0}\frac{(\|X\|^2-4H^2)H^2}{(\dist(z,z_0))^a}<\infty
\]
and
\[
\limsup_{z\rightarrow z_0}\frac{(\dist(z,z_0))^b}{\|\Phi\|^2}=\left(\liminf_{z\rightarrow z_0} \frac{\|\Phi\|^2}{(\dist(z,z_0))^b}\right)^{-1}<\infty.
\]
Since
\[
\frac{(\|X\|^2-4H^2)H^2}{\|\Phi\|^2} = \frac{(\|X\|^2-4H^2)H^2}{(\dist(z,z_0))^a}\cdot \frac{(\dist(z,z_0))^b}{\|\Phi\|^2}\cdot\frac{1}{(\dist(z,z_0))^{b-a}},
\]
then
\[
\frac{(\|X\|^2-4H^2)H^2}{\|\Phi\|^2}:= \vp^2,\ \vp\in L^p_{loc} \Leftrightarrow \frac{1}{(\dist(z,z_0))^{(b-a)/2}}\in L^p_{loc},\ p>2.
\]
Note that
\[
\frac{1}{(\dist(z,z_0))^{\beta}}\in L^1_{loc} \Leftrightarrow \beta<2.
\]
If $b>a,$ since $b-a<2,$ we can choose $2<p<\frac{4}{b-a}$ such that $(b-a)p<2.$ If $b<a,$ then $(b-a)p<2$ for every $p>0.$ Thus, by using the hypothesis, we have
\[
\frac{(\|X\|^2-4H^2)H^2}{\|\Phi\|^2}:=\vp^2,\ \mbox{for}\ \vp\in L^p_{loc}.
\]
The result then follows from Theorem \ref{Gauss-space-2}.
\end{proof}

For surfaces with parallel weighted mean curvature in the Gaussian space, we have 

\begin{corollary}\label{Cor-lambda-000}
Let $X:\Sigma\rightarrow(\mathbb{R}^{2+m},\lan\cdot,\cdot\ran,e^{-\|X\|^2/4}),\ m\geq1,$ be an immersion of a surface homeomorphic to the sphere into the Gaussian space. Assume that all the following assertions holds:
\begin{itemize}
\item[i)] $X$ has parallel weighted mean curvature ${\bf H}_f,$ i.e., $\n^\perp{\bf H}_f=0$.
\item[ii)] There exists an unitary normal vector field $\nu\in T\Sigma^\perp$ such that $\n^\perp\nu=0.$
\item[iii)] There exists a non-negative locally $L^p$ function $\vp:\Sigma\to\R,\ p>2,$ and a locally integrable function $G:[0,\infty)\rightarrow[0,\infty)$ satisfying $\limsup_{t\rightarrow 0}G(t)/t<\infty,$ such that  
\begin{equation}\label{Cor-lambda-0}
(\|X\|^2 -4\|{\bf H}_f - {\bf H}\|^2)\lan {\bf H},\nu\ran^2\leq\varphi^2 G(\|\Phi^\nu\|)^2,
\end{equation}
\end{itemize}

Then $X(\Sigma)$ is contained in a round hypersphere of $\R^{2+m}.$ Moreover, if ${\bf H}\neq0$ and $\nu={\bf H}/\|{\bf H}\|,$ then $X(\Sigma)$ is a minimal surface of a round hypersphere of $\R^{2+m}$ of radius 
\[
\sqrt{\lan{\bf H}_f,\nu\ran^2 + 4} - \lan{\bf H}_f,\nu\ran.
\]

Here $\|\Phi^\nu\|$ is the matrix norm of $\Phi^\nu=A^\nu-(\tr A^\nu/2)I,$ where $A^\nu$ is the shape operator of the second fundamental form of $X$ relative to $\nu,$ $\tr A^\nu$ is its trace, and $I:T\Sigma\to T\Sigma$ is the identity operator.
\end{corollary}
\begin{proof}
By taking $F(t)=t/4$ in (\ref{HG}), we obtain
\[
\frac{1}{4}\left|\lan{\bf H}_f,\nu\ran-\frac{1}{2}\lan X,\nu\ran\right|\|X^\top\|\leq\varphi G(\|\Phi^\nu\|),
\]
i.e.,
\[
\frac{1}{4}\left|\lan {\bf H},\nu\ran\right|\|X^\top\|\leq\varphi G(\|\Phi^\nu\|).
\]
Since 
\[
\|X^\top\| = \|X\|^2 - \|X^\perp\|^2 = \|X\|^2 - 4\|{\bf H}_f - {\bf H}\|,
\]
then (\ref{Cor-lambda-0}) becomes (\ref{HG}) and the result comes from Theorem \ref{Gauss-space}.

In order to determine the radius of the sphere, consider ${\bf H}_{\ss}$ the mean curvature vector of $\Sigma$ in $\ss^{1+m}(x_0,R),$ where $x_0$ is the center and $R$ is the radius of the sphere, and $\widetilde{II}$ the second fundamental form of $\ss^{1+m}(x_0,R)$ in $\R^{2+m}.$ We have
\[
{\bf H} = {\bf H}_{\ss}+\sum_{i=1}^2\widetilde{II}(e_i,e_i) \Rightarrow \lan{\bf H},\nu\ran= \lan{\bf H}_{\ss},\nu\ran + \frac{2}{R}
\]
where $\{e_1,e_2\}$ is an orthonormal frame of $T\Sigma.$ Since ${\bf H}_{\ss}\in T\ss^{1+m}(x_0,R),$ then $\lan{\bf H}_{\ss},\nu\ran=0,$ i.e.,
\begin{equation}\label{H-nu-2}
\lan{\bf H},\nu\ran=\frac{2}{R}.
\end{equation} 
Since ${\bf H}={\bf H}_f - \frac{1}{2}X^\perp,$ we have
\begin{equation}\label{X-nu}
\lan X,\nu\ran = 2\lan{\bf H}_f,\nu\ran - 2\lan{\bf H},\nu\ran = 2\lan{\bf H}_f,\nu\ran -\frac{4}{R}.
\end{equation}
On the other hand, $X(\Sigma)\subset\ss^{1+m}(x_0,R)$ implies $X=x_0-R\nu.$ This gives $\lan X,\nu\ran = \lan x_0,\nu\ran - R,$ i.e.,
\begin{equation}\label{x0-nu}
\lan x_0,\nu\ran = R + 2\left(\lan{\bf H}_f,\nu\ran - \frac{2}{R}\right).
\end{equation}
Taking the gradient in (\ref{X-nu}) and using that $\lan {\bf H}_f,\nu\ran$ is constant, we obtain 
\[
\lan X,R\nu\ran=\lan X,x_0\ran-\|X\|^2
\]
so that
\[
0=\n\lan X,R\nu\ran = \n\lan X,x_0\ran - \n(\|X\|^2)=x_0^\top - 2X^\top = -X^\top.
\]
Since $\n(\|X\|^2)=2X^\top=0,$ we deduce that $\|X\|^2$ is constant. Taking the Laplacian
\[
0=\frac{1}{2}\Delta\|X\|^2 = \lan{\bf H},X\ran + 2
\]
and using (\ref{H-nu-2}), we have
\[
\lan x_0,{\bf H}\ran = \lan X +R\nu,{\bf H}\ran = \lan X,{\bf H}\ran + R\lan\nu,{\bf H}\ran =  -2 + R\cdot\frac{2}{R}=0.
\]
If $\nu=\frac{{\bf H}}{\|{\bf H}\|},$ then (\ref{x0-nu}) becomes
\[
0=R + 2\left(\lan{\bf H}_f,\nu\ran - \frac{2}{R}\right),
\]
which gives
\[
R=\sqrt{\lan{\bf H}_f,\nu\ran^2+4}-\lan{\bf H}_f,\nu\ran.
\]
\end{proof}

In particular, for $\lambda$-surfaces, we obtain
\begin{corollary}

Let $X:\Sigma\rightarrow\R^3$ be a immersed $\lambda$-surface homeomorphic to the sphere. If there exists a non-negative locally $L^p$ function $\vp:\Sigma\to\R,\ p>2,$ and a locally integrable function $G:[0,\infty)\rightarrow[0,\infty)$ satisfying $\limsup_{t\rightarrow 0}G(t)/t<\infty,$ such that  
\[
\left(\|X\|^2-4(\lambda-H)^2\right)H^2\leq\varphi^2 G(\|\Phi\|)^2,
\]
then $X(\Sigma)$ is a round sphere of radius $\sqrt{\lambda^2+4}-\lambda$ and center at the origin.

Here  $\|\Phi\|$ denotes the matrix norm of $\Phi=A-(H/2)I,$ where $A$ is the shape operator of the second fundamental form of $X,$ $H$ is its non-normalized mean curvature, and $I$ is the identity operator of $T\Sigma.$
\end{corollary}

\begin{remark}
In the proof of Corollary \ref{Cor-lambda-000}, since the codimension can be $m\geq 2,$ we have that the spheres $\|X\|^2=constant$ and $\mathbb{S}^{1+m}(x_0,R)$ could be different. In this case we will have 
\[
X(\Sigma)\subset \mathbb{S}^{1+m}(x_0,R)\cap\mathbb{S}^{1+m}(0,\|X\|),
\]
where this intersection is, by its turn, a $m$-dimensional sphere.
\end{remark}

\section{Umbilical points in rotational Self-shrinkers and the Drugan's example}\label{sec.ex}

Our goal in this section is to show that the hypothesis (\ref{HG-2}) of Theorem \ref{Gauss-space-2} is necessary and cannot be removed. For that, we will study what happens in a neighborhood of certain umbilical point of a rotational self-shrinker, particularly the example given by Drugan in \cite{D}. If a smooth surface of revolution intersects the axis of rotation (perpendicularly), then the point in this intersection is an umbilical point. We remark that, since Drugan's example is homeomorphic to the sphere, it has two of these umbilical points. In this section we show that, if $\Sigma$ is a rotational self-shrinker which is not a plane nor a sphere, then 
\[
\frac{\sqrt{(\|X\|^2-4H^2)H^2}}{\|\Phi\|}\not\in L^p_{loc},\ \forall p>2,
\]
in a neighborhood of the umbilical point which intersects the rotation axis. By using this result, we can conclude that Drugan's self-shrinker is an example of self-shrinker homeomorphic to the sphere which does not satisfy the hypothesis (\ref{HG-2}) of Theorem \ref{Gauss-space-2}, proving that this hypothesis is necessary (see Figure \ref{fig4}).

\begin{figure}[ht]\label{fig4}
\centering
\includegraphics[scale=0.6]{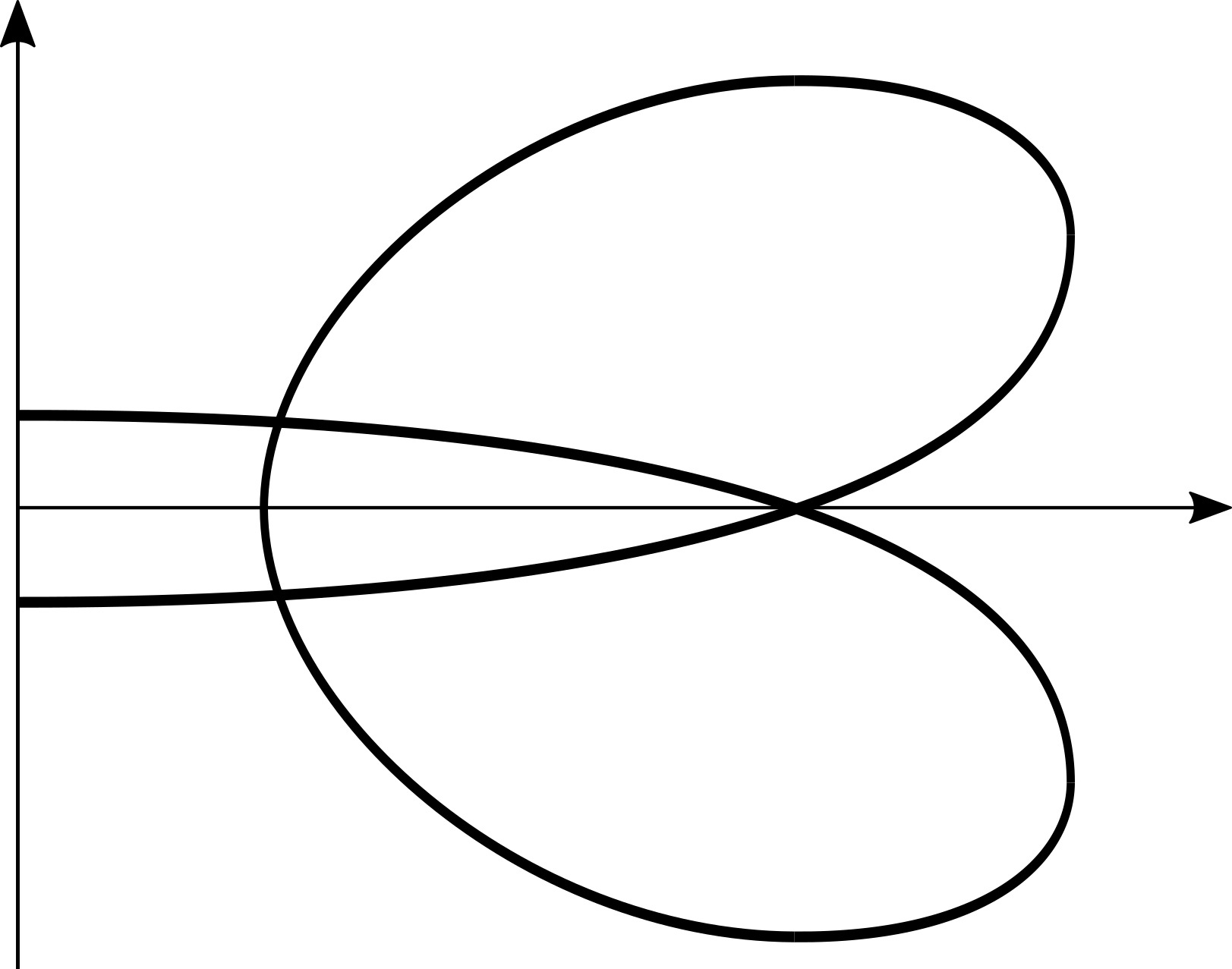}
\caption{Draft of the profile curve of the Drugan's example. The surface is obtained by rotating the profile curve around the vertical axis. The intersection of the profile curve with the rotation axis gives two isolated umbilical points which do not satisfy the hypothesis (\ref{HG-2}) of Theorem \ref{Gauss-space-2}}
\end{figure}

Let $\Sigma$ be a smooth rotational self-shrinker. If the profile curve is written as a graph $(x,\gamma(x)),$ the self-shrinker equation becomes
\begin{equation}\label{ODE-1}
\frac{\gamma''(x)}{1+(\gamma'(x))^2} = \left(\frac{x}{2} - \frac{1}{x}\right)\gamma'(x) - \frac{1}{2}\gamma(x).
\end{equation}
Since the principal curvatures of a rotational surface with profile curve $(x(t),y(t))$ are given by
\[
k_1=\frac{-y'(t)}{x(t)\sqrt{(x'(t))^2 + (y'(x))^2}}, \ k_2=\frac{x''(t)y'(t) - x'(t)y''(t)}{((x'(t))^2+(y'(t))^2)^{3/2}},
\]
we have
\[
k_1=\frac{-\gamma'(x)}{x\sqrt{1+(\gamma'(x))^2}} \ \mbox{and} \ k_2=\frac{-\gamma''(x)}{(1+(\gamma'(x))^2)^{3/2}}.
\]
This implies that, if the profile curve is a graph $(x,\gamma(x)),$ a point of the rotational surface is umbilical if and only if $k_1=k_2,$ i.e.,
\[
\frac{\gamma'(x)}{x} = \frac{\gamma''(x)}{1+(\gamma'(x))^2}.
\]
Thus, a point $(x,\gamma(x))$ of a profile curve of a self-shrinker gives an umbilic point if and only if
\[
\frac{\gamma'(x)}{x} = \left(\frac{x}{2} - \frac{1}{x}\right)\gamma'(x) - \frac{1}{2}\gamma(x),
\]
or, equivalently, 
\[
x\gamma(x) = (x^2-4)\gamma'(x).
\]
Define the function
\begin{equation}\label{F}
F(x) = x\gamma(x) - (x^2-4)\gamma'(x).
\end{equation}
A point $(x,\gamma(x))$ of a profile curve of a self-shrinker gives an umbilic point if and only if $F(x)=0.$ Moreover, 
\begin{equation}\label{gamma-phi}
\begin{aligned}
\|\Phi\|=\frac{1}{\sqrt{2}}|k_1-k_2|&=\frac{1}{\sqrt{2}\sqrt{1+(\gamma'(x))^2}}\left|-\frac{\gamma'(x)}{x} + \frac{\gamma''(x)}{1+(\gamma'(x))^2}\right|\\
&=\frac{1}{\sqrt{2}\sqrt{1+(\gamma'(x))^2}}\left|-\frac{\gamma'(x)}{x} + \left(\frac{x}{2}-\frac{1}{x}\right)\gamma'(x) - \frac{\gamma(x)}{2}\right|\\
&=\frac{1}{\sqrt{2}\sqrt{1+(\gamma'(x))^2}}\left|\left(\frac{x}{2}-\frac{2}{x}\right)\gamma'(x)-\frac{\gamma(x)}{2}\right|\\
&=\frac{|(x^2-4)\gamma'(x) - x\gamma(x)|}{2\sqrt{2}x\sqrt{1+(\gamma'(x))^2}}.\\
\end{aligned}
\end{equation}
On the other hand,
\begin{equation}\label{gamma-H}
-2H=\lan X,N\ran = \frac{(x,\gamma(x))\cdot(\gamma'(x),-1)}{\sqrt{1+(\gamma'(x))^2}}=\frac{x\gamma'(x)-\gamma(x)}{\sqrt{1+(\gamma'(x))^2}}
\end{equation}
and
\begin{equation}\label{gamma-X}
\begin{aligned}
\|X\|^2-4H^2 &= \|X\|^2 - \lan X,N\ran^2 = \frac{(x+\gamma(x)\gamma'(x))^2}{1+(\gamma'(x))^2},
\end{aligned}
\end{equation}
which implies
\[
\frac{\sqrt{(\|X\|^2-4H^2)H^2}}{\|\Phi\|} = \frac{\sqrt{2}x|x+\gamma(x)\gamma'(x)||\gamma(x)-x\gamma'(x)|}{\sqrt{1+(\gamma'(x))^2}|(x^2-4)\gamma'(x)-x\gamma(x)|}.
\]
For our purposes we will need the Taylor expansion of $\gamma$ and $F$ near zero.
\begin{lemma}\label{Taylor}
Let $\gamma(x)$ be the solution of (\ref{ODE-1}) with the initial conditions $\gamma(0)=b$ and $\gamma'(0)=0.$ Then, near $x=0,$ we have
\[
\gamma(x)=b - \frac{b}{8}x^2 - \frac{b}{256}\left(3+\frac{b^2}{4}\right)x^4 + O(x^5)
\]
and
\[
F(x)=-\frac{b}{16}\left(1+\frac{b^2}{4}\right)x^3 + O(x^4),
\]
where $F(x)=x\gamma(x)-(x^2-4)\gamma'(x),$ see (\ref{F}).
\end{lemma}
\begin{proof}
Let
\[
\gamma(x) = a_0 + a_1x + a_2x^2 + a_3x^3 + a_4x^4 + O(x^5).
\]
We have
\[
\gamma'(x) = a_1 + 2a_2 x + 3a_3 x^2 + 4a_4 x^3 + O(x^4)
\]
and
\[
\gamma''(x) = 2a_2 + 6a_3 x + 12 a_4 x^2 + O(x^3).
\]
Since $\gamma(0)=b,\ \gamma'(0)=0$ and $\gamma''(0)=-b/4,$ we obtain $a_0=b,$ $a_1=0,$ and $a_2=-b/8,$ which implies
\[
\gamma(x) = b - \frac{b}{8}x^2 + a_3x^3 + a_4x^4 + O(x^5).
\] 
In order to calculate $a_3$ and $a_4$ we will use equation (\ref{ODE-1}). Notice that
\[
\begin{aligned}
\left(\frac{x}{2}-\frac{1}{x}\right)\gamma'(x)&= \left(\frac{x}{2}-\frac{1}{x}\right)\left(-\frac{b}{4}x + 3a_3x^2 + 4a_4x^3 + O(x^4)\right)\\
&=\frac{b}{4} - 3a_3x - \left(\frac{b}{8}+4a_4\right)x^2 + O(x^3)
\end{aligned}
\]
implies
\[
\begin{aligned}
\left(\frac{x}{2}-\frac{1}{x}\right)\gamma'(x) - \frac{1}{2}\gamma(x)&=\frac{b}{4} - 3a_3x - \left(\frac{b}{8}+4a_4\right)x^2 + O(x^3) - \frac{b}{2} - \frac{b}{16}x^2 + O(x^3)\\
&=-\frac{b}{4} - 3a_3x - \left(\frac{3b}{16}+4a_4\right)x^2 + O(x^3).\\
\end{aligned}
\]
Since
\[
\begin{aligned}
1+(\gamma'(x))^2 &= 1 + x^2\left(-\frac{b}{4} + 3a_3x + 4a_4x^2 + O(x^3)\right)^2\\
& = 1 + \frac{b^2}{16}x^2 + O(x^3),\\
\end{aligned}
\]
we have
\[
\begin{aligned}
(1+(\gamma'(x))^2)&\left[\left(\frac{x}{2}-\frac{1}{x}\right)\gamma'(x) - \frac{1}{2}\gamma(x)\right]\\
& = \left(1+\frac{b^2}{16}x^2 + O(x^3)\right)\left(-\frac{b}{4} - 3a_3x - \left(\frac{3b}{16}+4a_4\right)x^2 + O(x^3)\right)\\
&=-\frac{b}{4} - 3a_3x - \left(\frac{3b}{16} + \frac{b^3}{64} + 4a_4\right)x^2 + O(x^3).
\end{aligned}
\]
By using (\ref{ODE-1}), p. \pageref{ODE-1},
\[
-\frac{b}{4} + 6a_3x + 12a_4x^2 + O(x^3) = -\frac{b}{4} - 3a_3x - \left(\frac{3b}{16} + \frac{b^3}{64} + 4a_4\right)x^2 + O(x^3),
\]
which implies
\[
a_3=0 \ \mbox{and} \ a_4=-\frac{b}{256}\left(3 + \frac{b^2}{4}\right).
\]
Thus, the Taylor expansion of $\gamma$ near zero is
\[
\gamma(x)=b - \frac{b}{8}x^2 - \frac{b}{256}\left(3+\frac{b^2}{4}\right)x^4 + O(x^5).
\]
Therefore
\[
\begin{aligned}
F(x)&=x\gamma(x) - (x^2-4)\gamma'(x)\\
& = x\left(b-\frac{b}{8}x^2 - \frac{b}{256}\left(3+\frac{b^2}{4}\right)x^4 + O(x^5)\right)\\
&\qquad - (x^2-4)\left(-\frac{b}{4}x - \frac{b}{64}\left(3+\frac{b^2}{4}\right)x^3 + O(x^4)\right)\\
&=-\frac{b}{16}\left(1+\frac{b^2}{4}\right)x^3 + O(x^4).\\
\end{aligned}
\]
\end{proof}

Now we present the first main result of this section.

\begin{proposition}
Let $\Sigma$ be a rotational self-shrinker which is not a plane or a sphere. If $z_0\in\Sigma$ is a umbilical point, then 
\begin{itemize}
\item[i)] $H(z_0)\neq 0.$
\item[ii)] $\|X(z_0)\|=2|H(z_0)|$ if and only if $z_0$ is in the rotation axis. Moreover, in this case
\[
\frac{\sqrt{(\|X\|^2-4H^2)H^2}}{\|\Phi\|}\not\in L^p_{loc},\quad \forall\, p>2.
\]
\end{itemize}
\end{proposition}
\begin{proof}
The proof will be based on the fact that any smooth curve in $\R^2$ is a union of graphs $y=\gamma(x)$ defined on intervals of the form $(-\infty,c_1], \ [c_1,c_2]$ or $[c_2,\infty),$ where $\gamma$ has a vertical tangent line in $c_1$ and $c_2,$ and vertical line segments.

Let $(a,\gamma(a)), a>0,$ be the point in the profile curve correspondent to $z_0.$ Since $z_0$ is umbilic, we have
\begin{equation}\label{gamma-a}
(a^2-4)\gamma'(a) = a\gamma(a).
\end{equation}
Thus, if $a\neq2,$ then
\begin{equation}\label{dif-a}
\gamma'(a)=\frac{a\gamma(a)}{a^2-4}.
\end{equation}

i) If, $a\neq2,$ then using (\ref{gamma-H}), p. \pageref{gamma-H}, we have
\[
|H(z_0)|=0 \Leftrightarrow |\gamma(a) - a\gamma'(a)|=\left|\gamma(a)-\frac{a^2\gamma(a)}{a^2-4}\right|=\frac{4|\gamma(a)|}{|a^2-4|}=0 \Leftrightarrow \gamma(a)=0.
\]
This implies that $\gamma'(a)=0$ and thus $\gamma(x)=0$ for every $x$ in a neighborhood of $a,$ by the uniqueness theorem for ordinary differential equations, i.e., $\Sigma$ is (a piece of) a plane. If $a=2,$ then, by (\ref{gamma-a}), we have $\gamma(2)=0.$ By using (\ref{gamma-H}), we have
\[
|H(z_0)|=\frac{|\gamma'(2)|}{\sqrt{1+(\gamma'(2))^2}}=0 \Leftrightarrow \gamma'(2)=0.
\]
But $\gamma(2)=\gamma'(2)=0$ gives again that $\gamma(x)=0,$ i.e., $\Sigma$ is (a piece of) a plane. Therefore, if $\Sigma$ is not (a piece of) a plane, we have that $H(z_0)\neq 0.$

ii) If $a\neq2,$ then using (\ref{gamma-X}), p. \pageref{gamma-X}, we have
\[
\begin{aligned}
\|X(z_0)\|=2|H(z_0)| &\Leftrightarrow |a+\gamma(a)\gamma'(a)|=0\\
& \Leftrightarrow \left|a+\frac{a(\gamma(a))^2}{a^2-4}\right|=a\left|1+\frac{(\gamma(a))^2}{a^2-4}\right|=0\\
&\Leftrightarrow a=0 \ \mbox{or} \ a<2 \ \mbox{and}\ \gamma(a)=\pm\sqrt{4-a^2}.
\end{aligned}
\]
In the second case, we have
\[
\gamma(a)=\pm\sqrt{4-a^2}\ \mbox{and, by (\ref{dif-a}),}\ \gamma'(a)= \mp\frac{a}{\sqrt{4-a^2}}.
\]
Since $\beta(x)=\pm\sqrt{4-x^2}$ is a solution of the self-shrinker equation with $\beta(a)=\gamma(a)$ and $\beta'(a)=\gamma'(a),$ then by the uniqueness theorem for ordinary differential equations, $\gamma(x)=\beta(x)$ in a neighborhood of $a$ and thus $\Sigma$ is (a piece of) $\ss^2(2).$ If $a=2,$ then $\gamma(2)=0.$ This implies that 
\[
\|X(z_0)\|^2-4(H(z_0))^2=\frac{4}{1+(\gamma'(2))^2}\neq 0.
\]
Therefore, if $\Sigma$ is not (a piece of) a sphere then $\|X(z_0)\|=2|H(z_0)|$ if and only if $a=0,$ i.e., $z_0$ is over the rotation axis. This concludes the first part of the proof of item ii).

In order to show that 
\[
\frac{\sqrt{(\|X\|^2-4H^2)H^2}}{\|\Phi\|}\not\in L^p_{loc},
\]
in the neighborhood of $a=0,$ consider the Taylor expansion of $\gamma(x)$ for $\gamma(0)=b>0$ and $\gamma'(0)=0$ given by Lemma \ref{Taylor}, p. \pageref{Taylor}:
\[
\gamma(x)=b-\frac{b}{8}x^2 - \frac{b}{256}\left(3+\frac{b^2}{4}\right)x^4 + O(x^5).
\]
By using (\ref{gamma-phi}), (\ref{gamma-H}), and (\ref{gamma-X}), p. \pageref{gamma-X}, we have
\[
\|\Phi\|=\frac{x^2\left|\frac{b}{16}\left(1+\frac{b^2}{4}\right)+O(x)\right|}{2\sqrt{2}\sqrt{1+O(x^2)}},
\]
\[
|H|=\frac{|b+O(x)|}{\sqrt{1+O(x^2)}} \ \mbox{and} \ \sqrt{\|X\|^2-4H^2}=\frac{\left|\left(1-\frac{b^2}{4}\right)x - \frac{b^2}{64}\left(1+\frac{b^2}{4}\right)x^3 + O(x^4)\right|}{\sqrt{1+O(x^2)}}.
\]
This implies
\[
\frac{\sqrt{(\|X\|^2-4H^2)H^2}}{\|\Phi\|}=\frac{1}{x}\left(\frac{2\sqrt{2}|b+O(x)|}{\sqrt{1+O(x^2)}}\right)\frac{\left|\left(1-\frac{b^2}{4}\right) - \frac{b^2}{64}\left(1+\frac{b^2}{4}\right)x^2 + O(x^3)\right|}{\left|\frac{b}{16}\left(1+\frac{b^2}{4}\right)+O(x)\right|}:=\frac{\widetilde{F}(x)}{x},
\]
where $\widetilde{F}(0)=32\sqrt{2}\left|1-b^2/4\right|\left(1+b^2/4\right)^{-1}.$ If $b=2,$ then, by the uniqueness theorem for ordinary differential equations, $\Sigma$ is (a piece of) $\ss^2(2).$ Thus we can consider $b\neq 2,$ which implies $\widetilde{F}(0)>0.$ In this case, we have, for sufficiently small $\ve>0,$
\[
\begin{aligned}
\int_{B_{z_0}(\ve)}\left(\frac{\sqrt{(\|X\|^2-4H^2)H^2}}{\|\Phi\|}\right)^pd\Sigma &= \int_0^\ve\int_0^{2\pi}\left(\frac{\widetilde{F}(x)}{x}\right)^p x\sqrt{1+(\gamma'(x))^2}d\theta dx\\
&\geq 2\pi \int_0^\ve\left(\frac{\widetilde{F}(x)}{x}\right)^p xdx\\
&\geq 2\pi (\widetilde{F}(0)-\delta)^p\int_0^\ve \frac{1}{x^{p-1}}dx = \infty, 
\end{aligned}
\]
for some $\delta=\delta(\ve)>0,$ since $p>2.$ This concludes the proof of item ii).
\end{proof}

\end{document}